\documentclass{article}

\usepackage{amsmath,amsthm,amssymb,graphicx,color}
\usepackage{subfigure,epsfig}
\usepackage{url}

\newcommand{\vb}{\mathbf{b}}
\newcommand{\vc}{\mathbf{c}}
\newcommand{\vd}{\mathbf{d}}
\newcommand{\ve}{\mathbf{e}}
\newcommand{\vf}{\mathbf{f}}
\newcommand{\vg}{\mathbf{g}}

\newcommand{\vn}{\mathbf{n}}
\newcommand{\vs}{\mathbf{s}}
\newcommand{\vt}{\mathbf{t}}
\newcommand{\vw}{\mathbf{w}}
\newcommand{\vu}{\mathbf{u}}
\newcommand{\vv}{\mathbf{v}}
\newcommand{\vx}{\mathbf{x}}

\newcommand{\vpsi}{\boldsymbol{\psi}}

\newcommand{\vzero}{\mathbf{0}}

\newcommand{\vH}{\mathbf{H}}
\newcommand{\vtH}{\widetilde{\mathbf{H}}}
\newcommand{\tH}{\widetilde{H}}
\newcommand{\vI}{\mathbf{I}}
\newcommand{\vL}{\mathbf{L}}

\newcommand{\calB}{\mathcal{B}}
\newcommand{\calN}{\mathcal{N}}
\newcommand{\calO}{\mathcal{O}}

\newcommand{\calC}{\mathcal{C}}
\newcommand{\calF}{\mathcal{F}}

\newcommand{\vcurl}{\textnormal{\textbf{curl}}}

\newcommand{\R}{\mathbb{R}}

\newcommand{\Pcfn}{P_{\vn}}

\newcommand{\RBFDFtangent}{\vs_\vf^\vt}
\newcommand{\RBFCFnormal}{\vs_\vf^\vn}
\newcommand{\RBFstreamtangent}{\vpsi_{\RBFDFtangent}}
\newcommand{\RBFpresstangent}{q_{\RBFDFtangent}}
\newcommand{\RBFstreamnormal}{\vpsi_{\RBFCFnormal}}
\newcommand{\RBFpressnormal}{q_{\RBFCFnormal}}

\newcommand{\ds}{\displaystyle}

\newcommand{\vertiii}[1]{{\left\vert\kern-0.25ex\left\vert\kern-0.25ex\left\vert #1 \right\vert\kern-0.25ex\right\vert\kern-0.25ex\right\vert}}

\newtheorem{lemma}{Lemma}
\newtheorem{theorem}{Theorem}
\newtheorem{proposition}{Proposition}
\newtheorem{remark}{Remark}

\vspace{-3cm}
\title{A Radial Basis Function Method for Computing Helmholtz-Hodge Decompositions\footnote{\today}}

\author{
\begin{tabular}{cc}
Edward J. Fuselier & Grady B. Wright\thanks{Research supported by grant
DMS-0934581, and DMS-0540779 from the National Science Foundation.} \\
Dept.\ of Mathematics and Computer Science & Dept.\ of Mathematics \\
High Point University & Boise State University \\
High Point, NC 27262 & Boise, ID 83725-1555
\end{tabular}
}

\begin{document}
\maketitle

\begin{abstract}
{A radial basis function (RBF) method based on matrix-valued kernels is presented and analyzed for computing two types of vector decompositions on bounded domains:  one where the normal component of the divergence-free part of the field is specified on the boundary, and one where the tangential component of the curl-free part of the field specified. These two decompositions can then be combined to obtain a full Helmholtz-Hodge decomposition of the field, i.e.\ the sum of divergence-free, curl-free, and harmonic fields.  All decompositions are computed from samples of the field at (possibly scattered) nodes over the domain, and all boundary conditions are imposed on the vector fields, not their potentials, distinguishing this technique from many current methods.  Sobolev-type error estimates for the various decompositions are provided and demonstrated with numerical examples.}
{Radial Basis Functions; Kernel Methods; Vector Decomposition; Divergence-free Approximation; Curl-free Approximation.}
\end{abstract}

\section{Introduction}
In the literature the phrases ``Helmholtz decomposition,'' ``Hodge decomposition,'' and ``Helmoltz-Hodge decomposition'' are used to describe a variety of vector decompositions in which a given field $\vf$ is written as a sum of divergence-free and curl-free fields. We will refer to any such decomposition as a Helmholz-Hodge decomposition (HHD). These decompositions are fundamental to many applications, from fluid dynamics and electromagnetics, to computer graphics and imaging. Each component plays an essential role in the underlying application. For example, the incompressible Navier-Stokes' equations describe the dynamics of an  incompressible fluid, the velocity field of the fluid is divergence-free while the (hydrostatic) pressure is curl-free.  This fact is exploited in projection methods, which are the dominant strategy employed for numerically solving these equations~\cite{Chorin1968,Temam1969}. A more general version of such a decomposition is given by the Hodge Theorem \cite{SchwarzHodgeBook}, which implies that vector fields $\vf$ on a compact domain $\Omega\subset\R^d$ can be split into the sum
$\vf = \vw + \nabla p + \nabla h$,
where $\vw$ is divergence-free and tangent to the boundary, $\nabla p$ is curl-free and normal to the boundary, and the scalar function $h$ is harmonic. This ``full'' HHD is used in graphics for detecting singularities (e.g. sinks, sources, and vortices) in vector fields that arise in various disciplines~\cite{PolthierPreuss2002}.

Several techniques exist to compute HHDs, with most making use of the vector field sampled on a mesh or grid.  The standard approach employed is to recast the problem in terms of a Poisson equation for a potential function $p$.   More specifically, given a vector field $\vf$, one numerically solves $\Delta p = \nabla \cdot \vf$, using, for example, finite difference or finite element methods. It follows then that $\vf$ is the sum of $\nabla p$ (which is curl-free) and $\vf - \nabla p$ (which is approximately divergence free).  One drawback of this approach is that in many applications it is not clear how to impose the correct boundary conditions on the Poisson problem for the potential $p$.  This is in part because the boundary conditions are typically imposed on the divergence-free or curl-free fields directly, not on the potentials for these fields.  For example, with regard to solving the incompressible Navier-Stokes equation, standard projection methods require a decomposition by calculating a pressure $p$ as the solution of a Poisson problem. However, the pressure does not have a boundary condition as it plays the role of a Lagrange multiplier, with its value being whatever it has to be to make the velocity field divergence-free~\cite{FLD:FLD598}.

Other techniques for decomposing vector fields use basis functions that are customized to split into analytically divergence- and curl-free parts.  These methods avoid having to explicitly solve a Poisson problem, but do require solving some other type of problem (e.g.\ an interpolation problem). Examples on periodic domains include those utilizing wavelets \cite{Deriaz_HelmholtzDecompWavelets}, and meshless kernel methods such as spherical basis functions \cite{Freeden_vectorsplines, FuselierWright2009:SphereDecomp}.  For domains with boundaries, a meshless radial basis function (RBF) method was developed for numerically solving certain static fluid problems (see \cite{WendlandScrader2011_Darcy,Wendland2009:DivFreeStokes}), with a by-product of this approach being a method for computing a certain type of decomposition.  

In this paper we develop and provide error estimates for a meshless RBF method for computing two standard vector decompositions on bounded domains in $\mathbb{R}^{d\geq 2}$: one where the normal component of the divergence-free part of the field specified on the boundary, and one where the tangential component of the curl-free part of the field is specified.  These decompositions can then be combined to compute the full HHD on a bounded domain.  Our approach utilizes matrix-valued RBFs that split into analytically divergence-free and curl-free parts. Each decomposition is obtained by solving a generalized interpolation problem, with the boundary conditions appearing on the velocity field variables and not on the potentials, and gives rise to a positive definite linear system of equations.  While we never work with the (vector and scalar) potentials of the components of the decomposed field directly, these potentials can be easily recovered at no added computational cost.  Our method provides accurate decompositions, but does require global information.  As such, a drawback, as is the case with many global kernel-based methods, is expense.  We hope this can be mitigated by employing approaches similar to those in the scalar kernel theory, such as using a multiscale approach \cite{WendlandFarrellMultiscaleRBFsPDEs} or by employing a localized basis~\cite{BeatsonBetterBasis,5manLocalBasis}, but this will be reported on separately.

As noted above the technique described in~\cite{WendlandScrader2011_Darcy,Wendland2009:DivFreeStokes} also gives rise to methods for computing certain vector decompositions in $\mathbb{R}^d$.   In fact, a vector decomposition as in Proposition 1 was obtained in \cite{WendlandScrader2011_Darcy}.  In these papers the authors use ``combined kernels'', which are constructed by incorporating a $d\times d$ divergence-free kernel with a scalar RBF to obtain a larger $(d+1)\times (d+1)$ kernel.   Our approach is different in that instead of combining kernels to make a larger one, we sum kernels with properties to match the HHD, which results in a diagonal $d\times d$ matrix-valued kernel. Though not obvious at first appearance, it can be shown that the techniques are in fact equivalent for a certain choice of the scalar kernel in the combined method. However, we approach the problem from a different perspective---instead of using a combined kernel that sets out to model the components of the vector field with separate kernels, we model the field directly with a single kernel that splits naturally. A practical by-product of this approach is that a large portion of the interpolation matrix becomes block-diagonal, which gives savings in terms of storage and computational efficiency. Where there is overlap in our work with previous work, we offer improvements in error estimates in terms of the order of approximation\footnote{Previous work derived estimates measured in the $H^1$ norm. We extend this to $L_2$, which gives an extra order of approximation.} and the domains on which they apply. We also include a vector decomposition not treated with kernel methods before (as described in Proposition 2) and develop the first kernel method for computing the full HHD.

The paper is organized as follows. Section \ref{prelims} contains the necessary preliminaries on function spaces and vector decompositions. In Section \ref{kernelprelims} we give background information on scalar and matrix-valued RBFs. Next, the construction of our kernel decompositions are described in detail in Section \ref{decompconstructions}. Error estimates and numerical experiments are presented in Sections \ref{mainerrorsection} and \ref{numerics}, respectively. We end the paper with some concluding remarks regarding decompositions with other boundary conditions.

\section{Preliminaries}\label{prelims}

We will distinguish between scalar and vector valued functions by denoting the latter in bold-face. We denote the gradient and divergence in the usual way, i.e. $\nabla$ and $\nabla\cdot$. The curl operator on three dimensional fields will be denoted by $\vcurl(\vf)$. Given a scalar valued function $f:\R^2\rightarrow \R$, we will use the same notation for $\vcurl(f):= (-\partial_yf,\partial_xf)$ --- this should cause no confusion. We will let $\Omega$ denote a connected open domain in $\R^d$ with boundary $\Gamma$ of H\"older class $\calC^{m,1}$ for some nonnegative integer $m$.

\subsection{Function spaces}

The function spaces we will work with are all Hilbert spaces: $L_2(\Omega)$ will denote the space of square integrable functions on $\Omega$, and $\vL_2(\Omega)$ will denote the space of all vector fields with in $L_2(\Omega)$. Given $s\geq 0$, we let $H^s(\Omega)$ denote the Sobolev class of functions on $\Omega$ with smoothness $s$, and denote its vectorial analogue by $\vH^s(\Omega)$. When the underlying domain is $\R^d$, we use the Fourier transform form of the inner product in these spaces. For example, the inner product on $\vH^s(\R^d)$ is given by
\begin{equation}\label{FTSobolevinnerproduct}
( \vf,\vg)_{\vH^s(\R^d)} := \int_{\R^d}{\overline{\widehat{\vf}(\omega)}^T}\widehat{\vg}(\omega)(1 + |\omega|^2)^{s}\,d\omega,
\end{equation}
where $\widehat{\vf}$ denotes the Fourier transform of $\vf$ and $|\omega|$ denotes the Euclidean length of $\omega\in\R^d$. We will also need the space of functions $\tH^s(\R^d)$, which is endowed with the inner product
\begin{equation}\label{Htildeinnerprod}
( f,g)_{\tH^s(\R^d)} := \int_{\R^d}\overline{\widehat{f}(\omega)}\widehat{g}(\omega)\frac{(1 + |\omega|^2)^{s+1}}{|\omega|^2}\,d\omega.
\end{equation}
It can be shown that $\tH^s(\R^d)$ is a subspace of $H^s(\R^d)$ and that $\|f\|_{H^s(\R^d)} \leq \|f\|_{\tH^s(\R^d)}$ for all $f\in\tH^s(\R^d)$ \cite[Proposition 2]{Fuselier2008:StabilityNative}. The space $\vtH^s(\R^d)$ is defined in an analogous way.

We denote the $L_2(\Gamma)$ inner product by $\langle\cdot,\cdot\rangle$. Sobolev spaces on the boundary $\Gamma$ can be defined in various ways. If the boundary is $\calC^{m,1}$, then to define $H^s(\Gamma)$ with $0\leq s\leq  m+1$ one can use charts and a partition of unity (see, for example \cite[Section 1.3.3]{Grisvard}). For $s\geq 0$, we let $H^{-s}(\Gamma)$ denote the dual space to $H^s(\Gamma)$, and the vector-valued cases for these spaces will be denoted in bold-face. 


Our arguments later will require standard operator interpolation on Sobolev spaces. A concise treatment of what we need can be found in \cite[Ch. 14]{BrennerScott_FEMbook}. For the interpolation arguments on boundary spaces, we will use the following fact from \cite[Theorem 7.7]{LionsMagenesVol1}: Let $0<\theta<1$. For all $s_1,s_2 \in \R$ with $s_1>s_2$ we have
\begin{equation}\label{boundaryinterpolationspace}
[H^{s_1}(\Gamma),H^{s_2}(\Gamma)]_{\theta} = H^{(1-\theta)s_1 + \theta s_2}(\Gamma),
\end{equation}
with equivalent norms, where $[H^{s_1}(\Gamma),H^{s_2}(\Gamma)]_{\theta}$ is the interpolation space with parameter $\theta$ between $H^{s_1}(\Gamma)$ and $H^{s_2}(\Gamma)$.

Lastly, we will make use of the following norms, which are both equivalent to $\|\cdot\|_{\vH^{s}(\Omega)}$ for all $s\geq 1$ when $\Gamma$ is at least $\calC^{\lceil s \rceil,1}$:
\begin{align}
\vertiii{\vu}^2_\vn&=\|\vu\|_{\vL_2(\Omega)}^2 + \|\vcurl(\vu)\|_{\vH^{s-1}(\Omega)}^2 + \|\nabla\cdot \vu\|_{H^{s-1}(\Omega)}^2 + \|\vu\cdot\vn\|_{H^{s-1/2}(\Gamma)}^2, \label{equivsobnorm_normal} \\
\vertiii{\vu}^2_\vt&=\|\vu\|_{\vL_2(\Omega)}^2 + \|\vcurl(\vu)\|_{\vH^{s-1}(\Omega)}^2 + \|\nabla\cdot \vu\|_{H^{s-1}(\Omega)}^2 + \|\vu\times\vn\|_{\vH^{s-1/2}(\Gamma)}^2\label{equivsobnorm_tan}.
\end{align}
For integer $s$, see \cite[Corollary 3.7, pg 56]{GiraultRaviart1986} for \eqref{equivsobnorm_normal}  and \cite[Proposition 6', pg. 237]{DautrayLion_vol3} and the proceeding remarks for \eqref{equivsobnorm_tan}. The fractional cases follow from standard interpolation arguments. Though stated here for $d=3$, similar results hold in the two dimensional case.
\subsection{Vector Decompositions}

The Helmholtz-Hodge decomposition for vector fields in $\vL_2(\R^d)$ can be easily described in terms of the Fourier transform. A field $\vf\in \vL_2(\R^d)$ is divergence-free if and only if $\omega^T\widehat{\vf}(\omega)=0$ almost everywhere, and $\vf$ is curl-free if and only if $\widehat{\vf}(\omega) = \omega \widehat{h}(\omega)$ for some $h\in H^1(\R^d)$. Letting $\calF^{-1}:\vL_2(\R^d)\rightarrow \vL_2(\R^d)$ denote the inverse Fourier transform, the operators
\begin{align}\label{eq;projection_operators}
P_{div}\vf := \calF^{-1}\left(\left(I - \frac{\omega\,\omega^T}{\|\omega\|^2}\right)\widehat{\vf}(\omega)\right),\mbox{\hspace{.5in}}P_{curl}\vf := \calF^{-1}\left(\left(\frac{\omega\,\omega^T}{\|\omega\|^2}\right)\widehat{\vf}(\omega)\right),
\end{align}
are projections on $\vL_2(\R^d)$, with $P_{div}\vf$ divergence-free, $P_{curl}\vf$ curl-free, and $P_{div}\vf\perp P_{curl}\vf$. With this, $\vf = P_{div}\vf + P_{curl}\vf$ uniquely decomposes $\vf$ into $\vL_2(\R^d)$-orthogonal divergence-free and curl-free fields. Further, $P_{div}$ and $P_{curl}$ are also orthogonal projections on any space whose inner product is of the form
\begin{equation}\label{weightedL2norm}
(\vf,\vg)_* = \int_{\R^d} \overline{\widehat{\vf}(\omega)}^T\widehat{\vg}(\omega) \,\varphi(\omega)\,d\omega,
\end{equation}
where the weight function $\varphi\geq 0$ is measurable---this includes all Sobolev spaces $\vH^s(\R^d)$ and $\vtH^s(\R^d)$. 
For fields on bounded domains we will focus on the two fundamental decompositions given in the following propositions.
\begin{proposition}\label{hodgeprop}
Let $\Omega\subset \R^d$ be a connected Lipschitz domain. $\vf\in \vL_2(\Omega)$ be such that $\nabla\cdot \vf \in \vL_2(\Omega)$, and let $g\in \vH^{-1/2}(\Gamma)$ satisfy $\langle g ,1 \rangle = 0$. Then one has the unique decomposition $\vf = \vw + \nabla p$, where $p\in H^1(\Omega)$, and $\vw \in \vL_2(\Omega)$ satisfies $\nabla\cdot \vw = 0$ with $\vw\cdot \vn = g $ on $\Gamma$. The function $p$ is uniquely determined up to a constant, and satisfies the bound
\begin{equation}\label{NeumannSolutionContinuity}
|p|_{H^1(\Omega)} = \|\nabla p\|_{L_2(\Omega)} \leq C\left(\|\nabla\cdot \vf\|_{L_2(\Omega)} + \| \vf\cdot\vn - g\|_{H^{-1/2}(\Gamma)}\right),
\end{equation}
where $C$ is some constant independent of $\vf$. When $g=0$, $\vw$ and $\nabla p$ are orthogonal in $\vL_2(\Omega)$.
\end{proposition}
\begin{proof}
Since the divergence of $\vf$ is in $L_2(\Omega)$, $\vf$ has a well-defined normal boundary component $\vf\cdot\vn\in H^{-1/2}(\Gamma)$ satisfying Green's formula (see \cite[Theorem 2.5]{GiraultRaviart1986}). Thus we can consider the following weak Neumann problem
\[(\nabla p,\nabla v) = (-\nabla\cdot \vf, v) + \langle \vf\cdot \vn - g,v\rangle\quad\forall\,v\in H^1(\Omega).\]
Standard Lax-Milgram theory dictates that the solution $p$ is continuous with respect to the data, giving \eqref{NeumannSolutionContinuity} (see, for example,~\cite[Proposition 1.2]{GiraultRaviart1986}). The field $\vw:= \vf - \nabla p$ has the other properties listed above. 
\end{proof}
%

\noindent An important by-product of this decomposition in the case $g = 0$ is the \emph{Leray projector} $P_L$ and its orthogonal complement $P_L^\perp$, defined by $P_L\vf := \vw$ and $P_L^{\perp}\vf: = \nabla p$. 

The next decomposition splits a vector field into a divergence-free field and a gradient field normal to the boundary. Note that $\nabla p$ is normal to the boundary if and only if $p|_\Gamma$ is constant on each of the connected components of $\Gamma$, which we denote by $\Gamma_0,\Gamma_1,\ldots,\Gamma_K$. The following is from Corollary $5'$ in \cite[pg 224]{DautrayLion_vol3}. 
\begin{proposition}\label{hodgeprop2}
Every $\vf\in \vL_2(\Omega)$ admits the unique orthogonal decomposition $\vf = \vw + \nabla p$, where $p\in H^1_c(\Omega) = \{v\in H^1(\Omega),\,v|_{\Gamma_i}=\mbox{constant},\,i=0,\ldots,K\}$. The vector field $\vw$ is divergence-free and perpendicular to $\nabla p$ in $\vL_2(\Omega)$. 
\end{proposition}

\subsubsection{Potential Functions and Extensions}
In what follows we require $\vw$ (the divergence-free term of $\vf$) to be expressed as $\vw=\vcurl(\vpsi)$ in the case of $d=3$ dimensions (or $\vw = \vcurl(\psi)$ when $d=2$).\footnote{Since our results will hold in two and three dimensions, throughout the remainder of the paper we will concentrate specifically on more complicated the $d=3$ case to avoid constantly distinguishing between these two cases.} We will also need a well-defined continuous assigment $\vw\rightarrow \vpsi$. This requires some mild assumptions on $\Omega$ in the event that $\Omega$ is multiply connected. Specifically, we assume that $\Omega$ can be made simply connected by a series of non-intersecting ``cuts'' $\Sigma_1,\ldots,\Sigma_n$, where $\Sigma_j\subset \Omega$ is a smooth variety (see for example \cite[pg. 217]{DautrayLion_vol3}). On such an $\Omega$, we have the following: 
\begin{proposition}\label{boundedvelpot}
A given $\vw\in \vL_2(\Omega)$ is an element of $\vcurl(\vH^1(\Omega))$ if and only if $\vw$ satisfies $\nabla\cdot \vw = 0$ and $\int_{\Gamma_i} \vw\cdot \vn\,d\Gamma = 0$ for all $i=0\ldots K$. Of all possible potential functions, there is a unique $\vpsi\in \vH^1(\Omega)$ such that $\vw = \vcurl(\vpsi)$ satisfying 
\begin{equation}\label{cutconditions}
\nabla\cdot \vpsi = 0,\quad \vpsi\cdot\vn = 0,\quad \langle \vpsi\cdot\vn,1\rangle_{\Sigma_i} = 0,\,\, i=1,\ldots,n.
\end{equation}
Finally, we have the bound $\|\vpsi\|_{\vH^1(\Omega)}\leq C\|\vw\|_{\vL_2(\Omega)}$ for some $C$ independent of $\vw$. 
\end{proposition}

\begin{proof}
The first claim is Corollary 4 from \cite[pg. 224]{DautrayLion_vol3}, and the unique assignment follows from Remark 4 proceeding the corollary. For continuity, note that $\vcurl(\vH^1(\Omega))$ endowed with the $\vL_2(\Omega)$ norm is closed \cite[pg. 222, Proposition 3]{DautrayLion_vol3}. Now let $V$ denote the subspace of fields $\vpsi\in\vL_2(\Omega)$ satisfying \eqref{cutconditions}. By \cite[pg. 225, Proposition 4]{DautrayLion_vol3}, $V$ is closed in $\vL_2(\Omega)$, so $V\cap \vH^1(\Omega)$ is closed in $\vH^1(\Omega)$. Using this one can show that the operator $T:\vcurl(\vH^1(\Omega))\rightarrow V\cap \vH^1(\Omega)$ given by $T\vw:=\vpsi$ is a closed map, and therefore continuous.
\end{proof}

This leads to potential functions for our decompositions that satisfy the following regularity result. 

\begin{proposition}\label{hodgeprop_potentials}
Let $\tau$ be such that $0\leq \tau \leq m$ and let $\vf\in \vH^\tau(\Omega)$. Then the decompositions in Propositions \ref{hodgeprop} and \ref{hodgeprop2} can be written as $\vf = \vcurl(\vpsi) + \nabla p,$
for uniquely determined potentials $p\in H^{\tau+1}(\Omega)$ and $\vpsi\in \vH^{\tau+1}(\Omega)$. For the decomposition in Proposition \ref{hodgeprop} with $g\in H^{\tau - 1/2}(\Gamma)$ satisfying $\langle g,1\rangle_{\Gamma_i}=0$ on each connected component of $\Gamma$, these potentials satisfy
\begin{equation}\label{potentialregularity}
\|p\|_{H^{\tau+1}(\Omega)} \leq C (\|\vf\|_{\vH^\tau(\Omega)} + \|g\|_{H^{\tau-1/2}(\Gamma)}),\quad \|\vpsi\|_{\vH^{\tau+1}}\leq C (\|\vf\|_{\vH^\tau(\Omega)} + \|g\|_{H^{\tau-1/2}(\Gamma)}),
\end{equation}
Similar bounds (with $g=0$) hold for the decomposition in Proposition \ref{hodgeprop2}.
\end{proposition}

\begin{proof}
Let $\tau$ be a nonnegative integer. In the case of Proposition \ref{hodgeprop}, with $g=0$, existence and uniqueness of $\vpsi$ follows from \cite[page 224, Corollary 5]{DautrayLion_vol3} and the proceeding remarks. The Proposition \ref{hodgeprop2} case follows from \cite[page 224, Corollary $5'$]{DautrayLion_vol3}. The additional regularity of the boundary gives regularity of these potentials (see, for example \cite[page 236, Corollary 7]{DautrayLion_vol3}). Recall that $V$ denotes the subspace of fields $\vpsi\in\vL_2(\Omega)$ satisfying \eqref{cutconditions}, and $V$ is closed in $\vL_2(\Omega)$, so $V\cap \vH^{\tau+1}(\Omega)$ is closed in $\vH^{\tau+1}(\Omega)$. From this one can show that the assignment $\vf\rightarrow \vpsi$ is a well-defined closed map, and thus obtain the bound for $\vpsi$ in \eqref{potentialregularity}. The scalar potential $p$ is unique if we require $\int_\Omega p\, dx = 0$. In a similar fashion as above, the bound for $p$ follows from the fact that the space $H^{\tau+1}(\Omega)\cap \{p\in L_2(\Omega)\,|\,\int_\Omega p\, dx = 0\}$ is closed in $H^{\tau+1}(\Omega)$. The fractional cases can be handled using standard interpolation arguments. 

To handle the case $g\neq 0$ from Proposition \ref{hodgeprop}, let $p_g$ be the solution of the problem
\[-\Delta p_g = 0\quad\mbox{in }\Omega,\quad \frac{\partial p_g}{\partial n} = -g\quad\mbox{on }\Gamma,\]
Note that that $\vw_g:=-\nabla p_g$ is divergence free. Since $\vw_g$ is divergence-free and $\vw_g\cdot \vn = g$ satisfies the conditions in Proposition \ref{boundedvelpot}, $\vw_g = \vcurl(\vpsi_g)$ for a unique $\vpsi_g$. Letting $\vf = \vcurl(\vpsi_0) + \nabla(p_0)$ denote the decomposition of $\vf$ from Proposition \ref{hodgeprop} with $g=0$, where the potentials are the unique potentials from above satisfying \eqref{potentialregularity} with $g=0$, the desired potentials are given by $\vpsi:=\vpsi_0 + \vpsi_g$ and $p:=p_0+p_g$.

The bound \eqref{potentialregularity} will follow from bounding $\vpsi_g$ and $p_g$. Since $g\in H^{\tau - 1/2}(\Gamma)$ and the domain is assumed smooth enough, we get the regularity bound \cite[Theorem 1.10]{GiraultRaviart1986} 
\begin{equation*}
\|\vw_g\|_{\vH^{\tau}(\Omega)} = \|\nabla p_g\|_{\vH^{\tau}(\Omega)} \leq \|p_g\|_{H^{\tau+1}(\Omega)} \leq C\|g\|_{H^{\tau-1/2}(\Gamma)}.
\end{equation*}
Using this with Proposition \ref{boundedvelpot}, $\vpsi_g$ satisfies the bound $\|\vpsi_g\|_{\vH^1(\Omega)}\leq C \|\vw_g\|_{\vL_2(\Omega)} \leq C\|g\|_{H^{-1/2}(\Gamma)}$. For higher regularity, we use \eqref{equivsobnorm_normal} with $s=\tau + 1$ to finish the proof:
\[ \|\vpsi_g\|_{\vH^{\tau + 1}(\Omega)}^2 \sim \vertiii{\vpsi_g}^2_{\vn} \leq C\left(\|\vw_g\|_{\vH^{1}(\Omega)}^2 + \|\vw_g\|_{\vH^{\tau}(\Omega)}^2\right) \leq C\|\vw_g\|_{\vH^{\tau}(\Omega)}^2 \leq C \|g\|_{H^{\tau-1/2}(\Gamma)}^2.\]
\end{proof}

\noindent We remark that the existence of these potentials is only used for theoretical purposes. The choice of cuts and the conditions \eqref{cutconditions} plays no role in implementing the kernel-based decomposition presented later. However, potential functions for each term in the kernel decomposition will be readily available.

Next we use these potentials to define an extension operator, which will be useful later. 

\begin{lemma}\label{lerayextensions}
Let $g\in H^{\tau - 1/2}(\Gamma)$ satisfy $\langle g,1\rangle_{\Gamma_i}=0$ on each connected component of $\Gamma$, and let $\vf = \vw + \nabla p$ denote the corresponding vector decomposition from Proposition \ref{hodgeprop}. Given $\Omega\subset\R^d$ satisfying the assumptions preceeding Proposition \ref{hodgeprop_potentials}, there exists an extension operator $E:\vH^\tau(\Omega)\rightarrow \vtH^\tau(\R^d)$, for all $\tau$ satisfying $0\leq \tau \leq m$, such that 
\begin{equation}\label{decompextension}
E\vf|_{\Omega} = \vf,\quad  P_{div}E\vf|_{\Omega} = \vw\quad \mbox {and  }P_{curl}E\vf|_{\Omega} = \nabla p,
\end{equation}
and is continuous in the sense that $\|E\vf\|_{ \vtH^\tau(\R^d)} \leq C\left(\|\vf\|_{\vH^\tau(\Omega)} + \|g\|_{H^{\tau - 1/2}(\Gamma)}\right).$
\end{lemma}

\begin{proof}
Let $p$ and $\vpsi$ denote the unique potentials for a given $\vf\in \vH^\tau(\Omega)$ in Proposition \ref{hodgeprop_potentials}. These can be extended using Stein's continuous extension $\mathfrak{E}:H^{\tau+1}(\Omega)\rightarrow H^{\tau+1}(\R^d)$, which we note is universal in the sense that $\mathfrak{E}$ does not depend on $\tau$ \cite[Chapter 4]{Stein1970}. We will interpret $\mathfrak{E}:\vH^{\tau+1}(\Omega)\rightarrow \vH^{\tau+1}(\R^d)$ as $\mathfrak{E}$ applied component-wise. We can then define the extension $E\vf := \vcurl(\mathfrak{E}\vpsi) + \nabla\mathfrak{E}p$, which satisfies \eqref{decompextension}. Lastly, \eqref{potentialregularity} gives us that $E$ is continuous: 
\begin{eqnarray*}
\|E\vf\|_{\vtH^\tau(\R^d)}^2& =& \int_{\R^d}\left(|\omega\times \widehat{\mathfrak{E}\vpsi}|^2 + |\omega\widehat{\mathfrak{E}p}|^2\right)\frac{(1 + |\omega|^2)^{\tau+1}}{|\omega|^2}\,d\omega \\
&\leq &\int_{\R^d}\left(|\widehat{\mathfrak{E}\vpsi}|^2 + |\widehat{\mathfrak{E}p}|^2\right)(1 + |\omega|^2)^{\tau+1}\,d\omega  = \|\mathfrak{E}\vpsi\|^2_{\vH^{\tau+1}(\R^d)} + \|\mathfrak{E}p\|^2_{H^{\tau+1}(\R^d)}\\
&\leq&  C\|\vpsi\|^2_{\vH^{\tau+1}(\Omega)} + C\|p\|^2_{H^{\tau+1}(\Omega)} \leq  C\left(\|\vf\|_{\vH^{\tau}(\Omega)} + \|g\|_{H^{\tau - 1/2}(\Gamma)}\right)^2.
\end{eqnarray*}
\end{proof}

\noindent These same arguments can be repeated to establish a continuous extension satisfying \eqref{decompextension} for the decomposition in Proposition \ref{hodgeprop2}.

\section{Radial Basis Functions and Related Kernels}\label{kernelprelims}


A kernel $\phi:\R^d\times\R^d\rightarrow \R$ is \emph{positive definite} if given any finite set of unique points $X = \{x_1,x_2,\ldots,x_N\}\subset \R^d$, the associated Gram matrix with entries $A_{ij} = \phi(x_i,x_j)$ is positive definite. The typical Ansatz for interpolation of function $f$ over the points $X$ with such a kernel is to find an interpolant of the form
\begin{equation}\label{pdinterpform}
s_f = \sum_{j=1}^N\phi(\cdot,x_j)c_j,
\end{equation}
where the coefficients $c_j$ are chosen so that $s_f\bigr|_X = f\bigr|_X$. Positive definiteness of the kernel ensures existence and uniqueness of the interpolant. If $\phi$ is radial in the sense that $\phi(x,y) = \varphi(|x - y|)$ for some univariate $\varphi$, then $\phi$ is a \emph{radial basis function} (RBF). It is common to simply write $\phi(x,y) = \phi(|x - y|)$. Good references on RBFs are, for example,~\cite{Buhmann_RBFbook,FasshauerMeshfreeBook,WendlandScatteredData}. 


For vector-valued approximations, there are matrix-valued kernels $\Phi:\R^d\times\R^d\rightarrow \R^d\times\R^d$.  Interpolants to a vector field $\vf:\R^d\rightarrow\R^d$ sampled at distinct points $X = \{x_1,x_2,\ldots,x_N\}\subset\R^d$ can be constructed from these kernels as follows:
\begin{equation}\label{generalMVkernelInterp}
\vs_\vf = \sum_{j=1}^N\Phi(\cdot,x_j)\vc_j,
\end{equation}
where the vector coefficients $\vc_j\in\R^d$ are chosen so that $\vs_\vf\bigr|_X = \vf\bigr|_X$.  This leads to the following $Nd\times Nd$ linear system of equations:
\begin{align}
\underbrace{
\begin{bmatrix}
\Phi(x_1,x_1) & \cdots & \Phi(x_1,x_N) \\
\vdots & \ddots & \vdots \\
\Phi(x_N,x_1) & \cdots & \Phi(x_N,x_N) 
\end{bmatrix}}_{\ds A}
\underbrace{
\begin{bmatrix}
\vc_1 \\
\vdots \\
\vc_N
\end{bmatrix}}_{\ds\vc}
=
\underbrace{
\begin{bmatrix}
\vf_1 \\
\vdots \\
\vf_N
\end{bmatrix}}_{\ds\vf}.
\label{eq:Phi_gram_matrix}
\end{align}
We say that $\Phi$ is positive definite if the Gram matrix $A$ in \eqref{eq:Phi_gram_matrix} is positive definite for any distinct set of points $X$. It will be useful later to express this property in a block-style quadratic form. Since $A$ is positive definite, we have
\begin{equation}\label{posdefquadform}
\sum_{j,k}\vc_k^T\Phi(x_k,x_j)\vc_j = \vc^T A \vc \geq 0, 
\end{equation}
with equality occurring if and only if $\vc_j = \vzero,\; j=1,\ldots,N$.

Customized matrix-valued kernels leading to divergence-free and curl-free approximations were introduced independently by several researchers in the 1990s: \cite{AmodeiBenbourhim_1991, Handscomb_1992_solenoidalTPS,NarcowichWard1994:GenHermite}. In all cases the construction of the customized kernel is fairly simple. For example, letting $\phi$ be an RBF on $\R^3$, we define
\begin{align}\label{customizedkerneldef}
\Phi_{div}(x,y)= \vcurl_x\,\vcurl_y\,\left(\phi(|x-y|)\vI\right)\mbox{\hspace{.2in}and\hspace{.2in}}\Phi_{curl}(x,y)= \nabla_x\nabla_y^T\left(\phi(|x-y|)\vI\right),
\end{align}
 where $\vI$ is the 3-by-3 identity matrix, the subscript in the differential operators indicate which argument they act on, and the curl of a matrix is interpretted as having the $\vcurl$ operator act on the matrix column-wise. Note that $\nabla_y\phi = -\nabla_x\phi$, so this simplifies to a form that readily generalizes to any $\R^d$:
\begin{align*}
\Phi_{div}(x,y):= (-\Delta \vI + \nabla\nabla^T)\phi(|x-y|)\mbox{\hspace{.2in}and\hspace{.2in}}\Phi_{curl}(x,y):= -\nabla\nabla^T\phi(|x-y|),
\end{align*}
where the differential operators act on $x$. It is easy to check that the second argument acts as a shift, e.g. $\Phi_{div}(x,y) = \Phi_{div}(x - y)$. If $\phi$ is positive definite, $\Phi_{div}$ and $\Phi_{curl}$ are both positive definite (see, for example \cite{Fuselier2008:StabilityNative,NarcowichWard1994:GenHermite}). Further, the kernel given by
\begin{equation}\label{fullkernel}
\Phi := \Phi_{div} + \Phi_{curl} = -\Delta\phi\vI
\end{equation}
is also positive definite because it is the sum of positive definite kernels. $\Phi$ decomposes naturally into its divergence-free and curl-free components. Indeed, given $x_j,\vc_j\in\R^d$, the identities\footnote{Here $\widehat{\phi}$ denotes the $d$-variate Fourier tranform of the single argument function $\phi(|\cdot|)$.}
\[\widehat{\Phi_{div}}(\omega) = \left(|\omega|^2\vI - \omega\omega^T\right)\widehat{\phi}(\omega)\quad\text{and}\quad \widehat{\Phi_{curl}}(\omega) = \left(\omega\omega^T\right)\widehat{\phi}(\omega)\]
imply that $P_{div}\Phi(\cdot,\vx_j)\vc_j = \Phi_{div}(\cdot,\vx_j)\vc_j$ and $P_{curl}\Phi(\cdot,\vx_j)\vc_j = \Phi_{curl}(\cdot,\vx_j)\vc_j$.

\subsection{The Native Space}

From here on out, we let $\Phi$ denote the matrix-valued kernel from \eqref{fullkernel}. Each positive definite matrix-valued kernel gives rise to a canonical reproducing kernel Hilbert space, commonly referred to as the \emph{native space} for that kernel. The native space for $\Phi$ is denoted by $\calN_\Phi(\R^d)$. A precise definition for $\calN_\Phi(\R^d)$ is not warranted here and we refer the interested reader to~\cite[Section 3]{Fuselier2008:StabilityNative}. $\Phi$ serves as a reproducing kernel in the sense that if $\vf$ is a vector field in $\calN_\Phi(\R^d)$ and $\vb\in \R^d$, then
\begin{equation}\label{repproperty}
(\vf,\Phi(\cdot,x)\vb)_{\calN_\Phi(\R^d)} = \vb^T\vf(x)\quad \forall\,x\in\R^d,
\end{equation}
where $(\cdot,\cdot)_{\calN_\Phi(\R^d)}$ denotes the inner product on $\calN_\Phi(\R^d)$. 

It can be shown that if $\phi\in C^2(\R^d)$ with $\Delta\phi\in L_1(\R^d)$, then the inner product in $\calN_\Phi(\R^d)$ is
\begin{equation}\label{generalNSinnerprod}
(\vf,\vg)_{\calN_\Phi(\R^d)} = \int_{\R^d}\frac{\overline{\widehat{\vf(\omega)}}^T\widehat{\vg}(\omega)}{|\omega|^2\widehat{\phi}(\omega)}\,d\omega,
\end{equation}
where $\widehat{\vf}$ is the Fourier tranform of $\vf$ and $\calN_\Phi(\R^d)\subset \vL_2(\R^d)$ is identified with all functions finite in the associated norm (see \cite[Section 3.1]{Fuselier2008:StabilityNative}). 
It immediately follows that if the RBF $\phi$ satisfies $\widehat{\phi}(\omega)\leq C(1 + |\omega|_2^2)^{-\tau-1}$ for some constant $C$, then $\calN_\Phi(\R^d)$ is continuously embedded in $\vtH^\tau(\R^d)$. If in addition 
\begin{equation}\label{algebraicdecay}
\widehat{\phi}(\omega)\sim(1 + |\omega|_2^2)^{-\tau - 1},
\end{equation}
then $\calN_\Phi(\R^d) = \vtH^\tau(\R^d)$ with equivalent norms.

\subsection{Generalized Interpolation}\label{generalizedinterp}

The reproducing kernel Hilbert space structure of the native space makes it possible to interpolate using a wide variety of continuous linear functionals. A concise treatment of this is given for scalar-valued RBFs in \cite[Chapter 16]{WendlandScatteredData}, and generalizes in a straightforward way to the matrix-valued case. We summarize the main results we need below.

Let $\Lambda \subset \calN_{\Phi}(\R^d)^*$ be a finite linearly independent collection of linear functionals, where $\calN_{\Phi}(\R^d)^*$ denotes the dual space to $\calN_{\Phi}(\R^d)$. Given the data $\{\lambda(\vf)\,|\,\lambda\in \Lambda\}$, where $\vf\in \calN_\Phi(\R^d)$, we look for a generalized interpolant to $\vf$ of the form
\[
\vs_\vf = \sum_{\lambda\in\Lambda}\vv_\lambda\alpha_\lambda,
\]
where $\alpha_\lambda\in \R$ and each $\vv_\lambda$ is the Riesz representer for $\lambda$. The interpolation conditions $\lambda(\vs_\vf) = \lambda(\vf)$ $\forall \,\lambda\in\Lambda$ lead to a linear system, and as long as the functionals are linearly independent the problem is uniquely solvable. Further, $\vs_\vf$ is perpendicular to $\vf - \vs_\vf$ in $\calN_{\Phi}(\R^d)$, which gives us the following:
\begin{equation}\label{bestapproxgeneralinterp}
\|\vf - \vs_\vf\|_{\calN_\Phi(\R^d)} \leq \|\vf\|_{\calN_\Phi(\R^d)}, \mbox{\hspace{.5in}} \|\vs_\vf\|_{\calN_\Phi(\R^d)} \leq \|\vf\|_{\calN_\Phi(\R^d)}.
\end{equation}
Note that since $\Phi$ is a reproducing kernel for $\calN_\Phi(\R^d)$, the Riesz representer for $\lambda$ can be written in terms of $\Phi$. For example, \eqref{repproperty} shows that the evaluation functional defined by $\lambda(f) = \vb^Tf(x_j)$ is represented in the native space as $\Phi(\cdot,x_j)\vb$. Next we consider functionals involving $P_{div}$.

\begin{proposition}\label{bdryfunctionalrep}
Let $x,\vn\in\R^d$, and define the functional $\nu(\vf):=\vn^TP_{div}\vf(x)$. Then $\nu$ is continuous on $\calN_\Phi(\R^d)$ and has Riesz representer $\Phi_{div}(\cdot,x)\vn$.
\end{proposition}

\begin{proof}
First note that by \eqref{generalNSinnerprod} and \eqref{weightedL2norm}, $P_{div}$ is a projection on $\calN_{\Phi}(\R^d)$. Using this and the reproducing kernel property of $\Phi$ we have 
\begin{eqnarray*}
|\nu(\vf)| &= & |(P_{div}\vf,\Phi(\cdot,x)\vn_j)_{\calN_{\Phi}(\R^d)}| \leq \|\Phi(\cdot,x)\vn\|_{\calN_{\Phi}(\R^d)} \|P_{div}\vf\|_{\calN_{\Phi}(\R^d)} \leq C\|\vf\|_{\calN_{\Phi}(\R^d)}.
\end{eqnarray*}
This gives us continuity. To verify the form of the representer, first note that the Fourier transform of $\vg:= \Phi_{div}(\cdot,x)\vn$ is given by
\[ \widehat{\vg}(\omega) = (|\omega|^2\vI - \omega\omega^T)\widehat{\phi}(\omega)e^{ix^T \omega}\vn.\] 
Using this and \eqref{generalNSinnerprod}, we have
\begin{eqnarray*}
(\vf,\vg)_{\calN_\Phi(\R^d)} & = & \vn^T \int_{\R^d}\left(\vI - \frac{\omega\omega^T}{|\omega|^T}\right)\widehat{\vf}(\omega)e^{ix^T \omega}\,d\omega  = \vn^T \int_{\R^d}\widehat{P_{div}\vf}(\omega)e^{ix^T \omega}\,d\omega = \vn^TP_{div}\vf(x).
\end{eqnarray*}
%
\end{proof}

\section{Kernel-based Decompositions}\label{decompconstructions}

In this section we show how to construct a kernel-based approximation to the decompositions discussed earlier. We will also show how one easily obtains potential functions from the kernel approximation. 

\subsection{Kernel Approximation with Divergence-free Boundary Conditions}\label{dfdecomp}

Given a target $\vf$ on $\Omega$ and boundary target $g$, it is our aim to construct a kernel approximation $\RBFDFtangent$ such that $P_{div}\RBFDFtangent$ and $P_{curl}\RBFDFtangent$, which we can compute analytically, approximate the appropriate terms of the decomposition in Proposition \ref{hodgeprop}.\footnote{We use the superscript $\vt$ because when $g=0$ the divergence-free portion is tangential to $\Gamma$.} We will construct our kernel-based vector decomposition by requiring full interpolation on nodes $X=\{x_1,x_2,\ldots,x_N\}\subset \Omega$, while at the same time enforcing boundary conditions at a dense set of nodes $Y = \{y_1,y_2,\ldots,y_M\}\subset \Gamma$. Although no repetition is allowed within each node set, $X$ and $Y$ can have a nonempty intersection.

Letting $\ve_i\in \R^d$ denote the vector whose only nonzero entry is a $1$ in the $i^\text{th}$ position, the interpolation functionals are given by $\lambda_j^{(i)}(\vf):=\ve_i^T \vf(x_j)$ for $1\leq i\leq d$, $x_j\in X$. The boundary functionals are given by $\nu_j(\vf):=\vn_{y_j}^TP_{div}\vf(y_j)$, $y_j\in Y$, where $\vn_y\in\R^d$ is the outward normal vector at $y\in \Gamma$. This gives a total of $dN + M$ conditions to be met. The basis functions to be used are the Riesz representers of these functionals, which from the previous section are given by $\Phi(\cdot,x_j)\ve_i$ and $\Phi_{div}(\cdot,y_j)\vn_{y_j}$, respectively.


Using these as basis functions, our RBF approximation will take the form
\begin{eqnarray}
\RBFDFtangent &=& \sum_{j=1}^{N}\sum_{i=1}^d\Phi(\cdot,x_j)\ve_ic_{ij} + \sum_{j=1}^{M}\Phi_{div}(\cdot,y_j)\vn_{y_j}d_j = \sum_{j=1}^{N}\Phi(\cdot,x_j)\vc_j + \sum_{j=1}^{M}\Phi_{div}(\cdot,y_j)\vn_{y_j}d_j, \label{eq:divfreeform}
\end{eqnarray}
where the coefficents $c_{ij}$, $1\leq i\leq d$ have been consolidated into the vector unknowns $\vc_j$ for each $j$, as in \eqref{generalMVkernelInterp}. Letting $\vf|_X$ denote the $dN\times 1$ vector whose $j^{\text{th}}$ $d\times 1$ block is given by $\vf(x_j)$, the interpolation conditions 1 and 2 above lead a linear system of the form
\begin{equation}\label{system}
\left[\begin{array}{cc}
             A & B \\
             B^T & C 
             \end{array}\right]\,
  \left[\begin{array}{c}
             \vc\\
             \vd 
             \end{array}\right]
= \left[\begin{array}{c}
             \vf|_{X}\\
             g |_{Y}
             \end{array}\right],
\end{equation}
where $A$ is the matrix given in \eqref{eq:Phi_gram_matrix}, $B$ is given by
\begin{align*}
B &= \left[\begin{array}{ccc}
             \Phi_{div}(x_1,y_1)\vn_{y_1} & \cdots & \Phi_{div}(x_1,y_M)\vn_{y_M} \\
             \vdots & \ddots & \vdots \\
             \Phi_{div}(x_N,y_1)\vn_{y_1} & \cdots & \Phi_{div}(x_N,y_M)\vn_{y_M} \\
             \end{array}\right], 
\end{align*}  
and $C$ is an $M\times M$ matrix given by $C_{ij} =\vn_{y_i}^T\Phi_{div}(y_i,y_j)\vn_{y_j}$. Note that due to the diagonal structure of the kernel $\Phi=\Delta\phi I$, the matrix $A$ can be rearranged to be block-diagonal, with $d$ identical $N\times N$ blocks along the diagonal. This not only reduces the cost of storing the interpolation matrix, but also makes it possible to solve \eqref{system} using a more efficient Schur complement method than if the matrix $A$ was dense~\cite{BenziSaddle}.

Note that the interpolation matrix in \eqref{system} is symmetric, and since we have taken the symmetric approach for generalized interpolation, it is also positive definite (and hence invertible) if the functionals involved are linearly independent \cite[Section 16.1]{WendlandScatteredData}. 
\begin{lemma}\label{functionallinind}
The functionals in $\Lambda=\{\lambda_{j}^{(i)}\,|\,x_j\in X,\,1\leq i\leq d\}\cup\{\nu_j\,|\,y_j\in Y\}$ are linearly independent. 
\end{lemma}

\begin{proof}
Suppose that some linear combination of the functionals in $\Lambda$ sums to zero. This is equivalent to its Riesz representer vanishing, i.e. 
\[\vg := \sum_{j=1}^{N}\Phi(\cdot,x_j)\vc_j + \sum_{l=1}^{M}\Phi_{div}(\cdot,y_l)\vd_l = \vzero,\]
where $\vd_l = \vn_l d_l$ for some scalars $d_l$. Since the terms in the decomposition $\vg = P_{div}\vg + P_{curl}\vg$ are orthogonal in $\calN_{\Phi}(\R^d)$, we have $\|P_{curl}\vg\|_{\calN_\Phi(\R^d)}^2 = \vzero$. We also have
\[\|P_{curl}\vg\|^2_{\calN_\Phi(\R^d)} = \sum_{j,k}(\Phi_{curl}(\cdot,x_j)\vc_j,\Phi_{curl}(\cdot,x_k)\vc_k)_{\calN_\Phi(\R^d)}.\]
Using the native space inner product \eqref{generalNSinnerprod} with the Fourier identities 
\[\widehat{\Phi_{curl}(\cdot,x_j)\vc_j} = (\omega\omega^T)\vc_j\widehat{\phi}(\omega)e^{ix_j^T\omega},\quad\quad \widehat{\Phi(\cdot,x_k)\vc_k} = \vc_k|\omega|^2\widehat{\phi}(\omega)e^{ix_k^T\omega},\]
it follows that 
\[
(\Phi_{curl}(\cdot,x_j)\vc_j,\Phi_{curl}(\cdot,x_k)\vc_k)_{\calN_\Phi(\R^d)} =(\Phi_{curl}(\cdot,x_j)\vc_j,\Phi(\cdot,x_k)\vc_k)_{\calN_\Phi(\R^d)}.
\]
Thus the reproducing property of $\Phi$ gives us 
\[\|P_{curl}\vg\|^2_{\calN_\Phi(\R^d)} = \sum_{j,k}(\Phi_{curl}(\cdot,x_j)\vc_j,\Phi(\cdot,x_k)\vc_k)_{\calN_\Phi(\R^d)} = \sum_{j,k}\vc_k^T\Phi_{curl}(x_k,x_j)\vc_j,\]
and since $\Phi_{curl}$ is positive definite \eqref{posdefquadform} implies that this equaling zero necessitates $\vc_j=0$ for all $j=1,\ldots,N$. Thus $\vg$ only consists of the boundary terms, i.e.
\[\vg = \sum_{l=1}^{M}\Phi_{div}(\cdot,y_l)\vd_l,\]
from which one can show similarly that 
\[
\|\vg\|_{\calN_{\Phi}(\R^d)}^2 = \sum_{l,m}\vd_l^T\Phi_{div}(y_l,y_m)\vd_m,\]
and since $\Phi_{div}$ is also positive definite we must have $\vd_l=\vzero$ for all $l=1,\ldots,M$. This completes the proof.
\end{proof}

Once \eqref{system} is solved, the resulting approximation decomposes as follows:
\begin{eqnarray*}
\RBFDFtangent &=& \underbrace{\sum_{j=1}^{N}\Phi_{div}(\cdot,x_j)\vc_j + \sum_{j=1}^{M}\Phi_{div}(\cdot,y_j)\vn_{y_j}d_j}_{P_{div}\RBFDFtangent} + \underbrace{\sum_{j=1}^{N}\Phi_{curl}(\cdot,x_j)\vc_j}_{P_{curl}\RBFDFtangent}. 
\end{eqnarray*}
As a bonus, we get a stream function $\RBFstreamtangent$ and velocity potential $\RBFpresstangent$ satisfying
\begin{equation}\label{interpolantpotentials}
\RBFDFtangent = \vcurl(\RBFstreamtangent) + \nabla \RBFpresstangent.
\end{equation}
Indeed, the identities \eqref{customizedkerneldef} imply that such potentials are given by
\[ \RBFstreamtangent := -\sum_{j=1}^{N}\vcurl(\phi(\cdot,x_j)\vc_j) -\sum_{j=1}^{M}\vcurl(\phi(\cdot,x_j)\vn_{y_j})d_j \quad\mbox{and}\quad \RBFpresstangent := -\sum_{j=1}^{N}\nabla^T(\phi(\cdot,x_j)\vc_j). 
\]
\subsection{Kernel Approximation with Curl-free Boundary Conditons}\label{cfdecomp}

We now focus on how to obtain a kernel-based approximation to the decomposition in Proposition \ref{hodgeprop2}, whose gradient term $\nabla p$ is normal to the boundary. As in the previous section, we enforce full interpolation on a node set $X$ and apply boundary conditions on a node set $Y$. The boundary conditions are imposed in this case by first projecting a kernel approximation $\RBFCFnormal$ onto the subspace of curl-free functions, and then setting all tangential components to zero pointwise. In $d=2$ dimensions, this is given by $\vt^T_{y_j}P_{curl}\RBFCFnormal(y_j) = 0$ for all $y_j\in Y$, where $\vt_{y_j}$ is tangent to $\Gamma$ at $y_j$. As before, the Riesz representers give the basis functions one should consider: for full interpolation they are the same as the previous section, and the boundary-centered basis functions are of the form $\Phi_{curl}(\cdot,y_j)\vt_{y_j}$. Thus the interpolant is written as
\begin{eqnarray}
\RBFCFnormal &=&\sum_{j=1}^{N}\Phi(\cdot,x_j)\vc_j + \sum_{j=1}^{M}\Phi_{curl}(\cdot,y_j)\vt_{y_j}d_j. \label{eq:curlfreenormalform}
\end{eqnarray}
In the $d=3$ case the two dimensional boundary leads to two basis functions at each shift on the boundary. For notational simplicity, we will continue with the $d=2$ case here. 

The interpolation constraints give rise to a linear system similar to \eqref{system} for determining the coefficients $\vc_j$ and $d_j$:
\begin{equation}\label{cfsystem}
\left[\begin{array}{cc}
             A & B \\
             B^T & C 
             \end{array}\right]\,
  \left[\begin{array}{c}
             \vc\\
             \vd 
             \end{array}\right]
= \left[\begin{array}{c}
             \vf|_{X}\\
             \vzero 
             \end{array}\right],
\end{equation}
where $A$ is the matrix given in \eqref{eq:Phi_gram_matrix}, B is given by
\begin{align*}
B &= \left[\begin{array}{ccc}
             \Phi_{curl}(x_1,y_1)\vt_{y_1} & \cdots & \Phi_{curl}(x_1,y_M)\vt_{y_M} \\
             \vdots & \ddots & \vdots \\
             \Phi_{curl}(x_N,y_1)\vt_{y_1} & \cdots & \Phi_{curl}(x_N,y_M)\vt_{y_M} \\
             \end{array}\right], 
\end{align*}
and $C$ is the $M\times M$ matrix with $C_{ij}=\vt_{y_i}^T\Phi_{curl}(y_i,y_j)\vt_{y_j}$.It can be shown using an argument similar to that in Lemma \ref{functionallinind} that the linear functionals involved are linearly independent, which guarantees that the matrix in \eqref{cfsystem} is symmetric and positive definite. The decomposition of the resulting kernel approximation is given by:
%
\begin{eqnarray*}
\RBFCFnormal &=& \underbrace{\sum_{j=1}^{N}\Phi_{div}(\cdot,x_j)\vc_j}_{P_{div}\RBFCFnormal} + \underbrace{\sum_{j=1}^{N}\Phi_{curl}(\cdot,x_j)\vc_j + \sum_{j=1}^{M}\Phi_{curl}(\cdot,y_j)\vt_{y_j}d_j}_{P_{curl}\RBFCFnormal}.
\end{eqnarray*}

In Section \ref{dferrorsection} we will show that $P_{div}\RBFCFnormal$ and $P_{curl}\RBFCFnormal$ approximate the terms from Proposition \ref{hodgeprop2}. Also one can use the form of the kernels \eqref{customizedkerneldef} to access potential functions $\RBFstreamnormal$ and $\RBFpressnormal$.


\section{Error Estimates}\label{mainerrorsection}

Our analysis follows the paradigm of RBF error estimates developed in recent years, where bounds on Sobolev functions having many zeros (the so-called ``zeros lemmas,'' or ``sampling inequalities'') play a prominent role \cite{NarcowichEtAl2005:ZerosLemma}. We will review the specific results we require below, and extend them slightly to suit our purposes. Next, we derive the error estimates in Sections \ref{dferrorsection} and \ref{cferrorsection}.

\subsection{Zeros Lemmas}
The zeros lemmas involve bounding the norm of Sobolev functions that vanish on a set $X = \{x_1,\ldots,x_N\}\subset \Omega\subset\R^d$ in terms of the density of $X$ in $\Omega$, which is quanitfied by the \emph{mesh norm}:
\[h_{\Omega} := \sup_{x\in\Omega}\text{dist}(x,X).\]
The following is from \cite{NarcowichEtAl2005:ZerosLemma}, with improvements in \cite[Theorem 4.6]{Wendland2009:DivFreeStokes}.  
\begin{proposition}\label{prop:zeroslemma}
Let $\Omega\subset \R^d$ be a bounded domain with Lipschitz boundary. Let $s\in \R$ with $s > d/2$, and let $\mu\in \R$ satisfy $0\leq \mu \leq s$. Also, let $X\subset\Omega$ be a discrete set with mesh norm $h_{\Omega}$ sufficiently small. Then there is a constant depending only on $\Omega$ such that if $h_{\Omega}\leq C_{\Omega}$ and if $u\in H^s(\Omega)$ satisfies $u|_{X}=0$, then
\begin{equation}
\|u\|_{H^\mu(\Omega)}\leq Ch_{\Omega}^{s-\mu}\|u\|_{H^s(\Omega)},
\end{equation}
where the constant $C$ is independent of $h_{\Omega}$ and $u$.
\end{proposition}
\noindent This result can also be extended to manifolds in a straightforward way (see \cite[Lemma 10]{FuselierWright_restrictedkernels}). Thus, if $u\in H^{s}(\Gamma)$ satisfies $u|_{Y}=0$, for $0\leq \mu\leq s$ one has
\begin{equation}\label{eq:zerosbdry}
\|u\|_{H^{\mu}(\Gamma)}\leq Ch_{\Gamma}^{s-\mu}\|u\|_{H^s(\Gamma)}.
\end{equation}
%
\noindent Here the mesh norm $h_\Gamma$ for a finite set $Y\subset \Gamma$, is defined just as in the Euclidean case, the only difference being that distances are measured on the surface $\Gamma$.

Note that the proposition above, the smoothness in the norm on the right-hand-side of the estimate is assumed to be high-enough so that the associated space of functions is continuous. However, such estimates hold for continuous functions in rougher norms, that is, if $s > \max\{d/2,1\}$ and $u\in H^s(\Omega)$ satisfies $u|_{X}=0$, then\footnote{The proof of Proposition \ref{prop:zeroslemma} involves local polynomial approximations on patches - in this case the polnomials are simply constants, which greatly simplifies the arguments.}
\begin{equation*}
\|u\|_{L_2(\Omega)}\leq Ch_{\Omega}|u|_{H^1(\Omega)}.
\end{equation*}
If the underlying domain is a surface, by applying this estimate on patches, we get the following for continuous functions $u:\Gamma\rightarrow \R$ with zeros on $Y\subset \Gamma$:
\begin{equation}\label{boundaryzerosH1}
\|u\|_{L_2(\Gamma)} \leq Ch_\Gamma|u|_{H^1(\Gamma)}.
\end{equation}

%
%
%

Lastly, in what follows we will need zeros estimates in negative-indexed Sobolev norms. Note that if $u\in C(\Gamma)\cap H^1(\Gamma)$, then obviously $u\in H^{1}(\Gamma)\subset H^{-1}(\Gamma)$. Thus we get
\[\|u\|_{H^{-1}(\Gamma)} = \sup_{\|\varphi\|_{H^{1}(\Gamma)} = 1} \langle u,\varphi\rangle= \|u\|_{L_2(\Gamma)}^2 / \|u\|_{H^{1}(\Gamma)}, \]
where since $u\in H^1(\Gamma)$ the supremum is achieved by choosing $\varphi = u/\|u\|_{H^{1}(\Gamma)}$. Thus if $u$ vanishes on $Y$, then with \eqref{boundaryzerosH1} we obtain
\begin{equation}\label{boundaryzerosHneg1}
\|u\|_{H^{-1}(\Gamma)} = \|u\|_{L_2(\Gamma)}^2 / \|u\|_{H^{1}(\Gamma)} \leq Ch_\Gamma\|u\|_{L_2(\Gamma)}.
\end{equation}

\subsection{Convergence with Divergence-free Boundary Conditions}\label{dferrorsection}

For the rest of the paper we assume that the RBF $\phi$ is such that $\mathcal{N}_{\Phi}(\R^d) = \vtH^\tau(\Omega)$ with equivalent norms, the boundary $\Gamma$ is smooth (at least $\mathcal{C}^{m,1}$ with $0<\tau\leq m$), and that the mesh norms for the node sets $X$ and $Y$ ($h_{\Omega}$ and $h_{\Gamma}$) are sufficiently small for the zeros lemmas to be applied. Further, we assume that $g$ satisfies the condition $\langle g,1\rangle_{\Gamma_i} = 0$ on each connected component of $\Gamma$. We begin with a basic interpolation estimate.

\begin{lemma}\label{lemma:fullapprox}
Let $\mu$ satisfy $0\leq \mu\leq \tau$. Let $\RBFDFtangent$ be the kernel approximation discussed in Section \ref{dfdecomp} for a given $\vf$ and $g$. Then for all $\vf\in \vH^\tau(\Omega)$ and $g\in H^{\tau-1/2}(\Gamma)$ we have\footnote{Here and throughout, $C$ is a constant independent of $\vf$, $g$, and the node sets.}
\[ \|\vf - \RBFDFtangent\|_{\vH^\mu(\Omega)} \leq C h_{\Omega}^{\tau-\mu}\left(\| \vf\|_{\vH^\tau(\Omega)} + \|g\|_{H^{\tau-1/2}(\Gamma)}\right).\]
\end{lemma}

\begin{proof}
Since $\vf -\RBFDFtangent$ has zeros on $X$, we may apply Proposition \ref{prop:zeroslemma} to get
\[\|\vf - \RBFDFtangent\|_{\vH^\mu(\Omega)} \leq Ch_{\Omega}^{\tau-\mu}\|\vf - \RBFDFtangent\|_{\vH^\tau(\Omega)}.\] 
Now we use the extension operator. Since $E\vf|_\Omega = \vf$ and $(P_{div}E\vf)|_{\Omega} = \vw$, where $\vw$ satisfies $\vw\cdot\vn = g$, then the data in the system used to determine $\vs^{\vt}_{E\vf}$ (see \eqref{system}) is the same as that of $\RBFDFtangent$. Thus we get $\RBFDFtangent = \vs^{\vt}_{E\vf}$. This with \eqref{bestapproxgeneralinterp}, the fact that $\vtH^\tau(\R^d)$ is norm equivalent to $\calN_\Phi(\R^d)$, and and the continuity of $E$ gives
\begin{eqnarray*}
\|\vf - \RBFDFtangent\|_{\vH^\tau(\Omega)} & = &\|E\vf - \vs^{\vt}_{E\vf}\|_{\vH^\tau(\Omega)} \leq \|E\vf - \vs^{\vt}_{E\vf}\|_{\vtH^\tau(\R^d)} \leq C\|E\vf - \vs^{\vt}_{E\vf}\|_{\calN_\Phi(\R^d)}\\
&\leq& C\|E\vf\|_{\calN_\Phi(\R^d)} \leq C\|E\vf\|_{\vtH^\tau(\R^d)} \leq C\left(\|\vf\|_{\vH^\tau(\Omega)} +\|g\|_{H^{\tau - 1/2}(\Gamma)}\right).
\end{eqnarray*}
This completes the proof.
\end{proof}

We continue our our analysis by showing that $P_{div}\RBFDFtangent\cdot \vn - g$ is small on the boundary.
\begin{lemma}\label{lemma:bdryapprox}
Let $\mu$ satisfy $0\leq \mu\leq \tau$. For all $\vf\in \vH^\tau(\Omega)$ and $g\in H^{\tau-1/2}(\Gamma)$ we have
\[ \|P_{div}\RBFDFtangent\cdot \vn - g \|_{H^{\mu-1/2}(\Gamma)} \leq C h^{\tau-\mu}_\Gamma \left(\|\vf\|_{\vH^\tau(\Omega)} + \|g\|_{H^{\tau - 1/2}(\Gamma)}\right).\]
\end{lemma}

\begin{proof}
First assume that $\mu\geq 1/2$. Recall that $P_{div}\RBFDFtangent\cdot\vn = g$ on the node set $Y\subset \Gamma$ by construction. Since the normals are assumed smooth and $\mu-1/2\geq 0$, we can apply \eqref{eq:zerosbdry} to get
\begin{eqnarray*}
\|P_{div}\RBFDFtangent\cdot \vn - g\|_{H^{\mu-1/2}(\Gamma)} &\leq &C h^{\tau-\mu-1/2}_{\Gamma}\|P_{div}\RBFDFtangent\cdot \vn - g\|_{H^{\tau-1/2}(\Gamma)}\\
& \leq & Ch_{\Gamma}^{\tau -\mu - 1/2}\left(\|P_{div}\RBFDFtangent\|_{\vH^{\tau- 1/2}(\Gamma)} + \|g\|_{H^{\tau - 1/2}(\Gamma)} \right).
\end{eqnarray*}
Applying the Trace Theorem and the fact that the $\vtH^{\tau}(\R^d)$ norm bounds the $\vH^{\tau}(\R^d)$ norm gives us
\[ 
\|P_{div}\RBFDFtangent\|_{\vH^{\tau-1/2}(\Gamma)} \leq C\|P_{div}\RBFDFtangent\|_{\vH^\tau(\Omega)} \leq C\|P_{div}\RBFDFtangent\|_{\vtH^\tau(\R^d)} = C\|P_{div}\vs^\vt_{E\vf}\|_{\vtH^\tau(\R^d)} \leq C \|\vs^\vt_{E\vf}\|_{\vtH^\tau(\R^d)},
\]
where in the last two steps we used the fact that $\RBFDFtangent = \vs^\vt_{E\vf}$ and that $P_{div}$ is a projection on $\vtH^{\tau}(\R^d)$.
The continuous embedding of $\calN_\Phi(\R^d)$ into $\vtH^\tau(\R^d)$, the bounds \eqref{bestapproxgeneralinterp}, and continuity of $E$ gives us
\begin{eqnarray*}
\|\vs^\vt_{E\vf}\|_{\vtH^\tau(\R^d)} &\leq &C\|\vs^\vt_{E\vf}\|_{\calN_{\Phi}(\R^d)} \leq C\|E\vf\|_{\calN_{\Phi}(\R^d)} \leq C\|E\vf\|_{\vtH^\tau(\R^d)} \leq C\|\vf\|_{\vH^\tau(\Omega)}.
\end{eqnarray*}
This gives us the correct approximation orders down to $\mu=1/2$. To get the estimates for $0\leq \mu\leq 1/2$, we will measure the error in the $H^{-1}(\Gamma)$ norm, and then obtain the desired bound by interpolation.

Let $\mathcal{V} = \{v\in H^{\tau-1/2}(\Gamma)\,:\, \langle v,1\rangle|_{\Gamma_i} = 0\,\,,\,\, 0\leq i\leq K\}$, and note that this space is closed in the $H^{\tau - 1/2}(\Gamma)$ norm. Next consider the Banach space $\calB:=\vH^{\tau}(\Omega)\times \mathcal{V}$ with obvious norm $\|(\vf,g)\|_{\calB}:= \|\vf\|_{\vH^{\tau}(\Omega)} + \|g\|_{H^{\tau -1/2}(\Gamma)}$. Now define the linear map $T:\calB\rightarrow L_2(\Gamma)$ given by $T(\vf,g):= P_{div}\RBFDFtangent\cdot\vn - g $. The argument above shows that 
\[\|T\|_{\calB\rightarrow L_2(\Gamma)} \leq Ch_\Gamma^{\tau-1/2}.\]
Similarly, considering $T$ as a map from $\calB$ to $H^{-1}(\Gamma)$, the zeros estimate \eqref{boundaryzerosHneg1} applies to the same arguments above to yield \[\|T\|_{\calB\rightarrow H^{-1}(\Gamma)} \leq Ch_\Gamma^{\tau+1/2}.\] Estimates for the space $H^{\mu-1/2}(\Gamma)$ now follow from interpolation theory. Specifically, the identity for interpolation spaces in \eqref{boundaryinterpolationspace} with $\theta = 1/2 - \mu$ gives us that $[L_2(\Gamma),H^{-1}(\Gamma)]_{1/2-\mu} = H^{\mu-1/2}(\Gamma)$. Interpolation of operators (see, for example \cite[Proposition 14.1.5]{BrennerScott_FEMbook}) tells us that $T$ maps $\calB$ into $H^{\mu-1/2}(\Gamma)$ with norm:
\[\|T\|_{\calB\rightarrow H^{\mu-1/2}(\Gamma)} \leq \|T\|_{\calB\rightarrow L_2(\Gamma)}^{1 - (1/2 - \mu)}\|T\|_{\calB\rightarrow H^{-1}(\Gamma)}^{(1/2 - \mu)} \leq  C h_\Gamma^{\tau-\mu}.\] 
This finishes the proof.
\end{proof}

Next, apply Proposition \ref{hodgeprop} to obtain $\RBFDFtangent=\vw_{\RBFDFtangent} + \nabla p_{\RBFDFtangent}$. Next we show that $P_{div}\RBFDFtangent$ approximates $\vw_{\RBFDFtangent}$. 

\begin{lemma}\label{lemma:projapprox}
Let $0\leq \mu\leq \tau$. For all $\vf\in \vH^\tau(\Omega)$ and $g\in H^{\tau-1/2}(\Gamma)$ we have
\[\|P_{div}\RBFDFtangent - \vw_{\RBFDFtangent}\|_{\vH^\mu(\Omega)} = \|P_{curl}\RBFDFtangent - \nabla p_{\RBFDFtangent}\|_{\vH^\mu(\Omega)}\leq C h_\Gamma^{\tau-\mu}\left(\| \vf\|_{\vH^{\tau}(\Omega)} + \|g\|_{H^{\tau - 1/2}(\Gamma)}\right).\]
\end{lemma}

\begin{proof}
The first equality follows easily from fact that $P_{div}\RBFDFtangent - \vw_{\RBFDFtangent} = \nabla p_{\RBFDFtangent} -  P_{curl}\RBFDFtangent$. For the rest, note that  $P_{curl}\RBFDFtangent = \nabla \RBFpresstangent$, where $\RBFpresstangent$ is from \eqref{interpolantpotentials}. It follows that 
\[P_{div}\RBFDFtangent = \vw_{\RBFDFtangent} + \nabla(p_{\RBFDFtangent} - \RBFpresstangent),\]
which is the decomposition in Proposition \ref{hodgeprop} applied to the function $\vf = P_{div}\RBFDFtangent$. Letting $\vv:=P_{div}\RBFDFtangent - \vw_{\RBFDFtangent}$, by \eqref{NeumannSolutionContinuity} we get the bound
\[
\|\vv\|_{\vL_2(\Omega)} = \|\nabla (p_{\RBFDFtangent} - \RBFpresstangent)\|_{\vL_2(\Omega)} \leq C \|P_{div}\RBFDFtangent\cdot\vn - g \|_{H^{-1/2}(\Gamma)}.
\]
An application of Lemma \ref{lemma:bdryapprox} finishes the proof for the $\mu=0$ case. For $\mu\geq 1$, we can use \eqref{equivsobnorm_normal} to get 
\[ \|\vv\|_{\vH^\mu(\Omega)}^2 \sim \vertiii{\vv}^2_\vn = \|\vv\|_{\vL_2(\Omega)}^2 + \|\vv\cdot\vn\|_{H^{\mu-1/2}(\Gamma)}^2,\]
where we used the fact that $\vv$ is divergence-free and curl-free. After applying the bound on $\|\vv\|_{\vL_2(\Omega)}$ above and the fact that $\vv\cdot \vn = P_{div}\RBFDFtangent\cdot\vn - g$, we get
\[\|\vv\|_{\vH^\mu(\Omega)} \leq C \|P_{div}\RBFDFtangent\cdot\vn - g\|_{H^{\mu-1/2}(\Gamma)}.\]
Another application of Lemma \ref{lemma:bdryapprox} finishes the proof for $1\leq \mu\leq \tau$. The $0<\mu<1$ case can be handled by interpolating the operator $T$ between the ranges $\vL_2(\Omega)$ and $\vH^1(\Omega)$, where $T$ is given by $T(\vf,g):= P_{div}\RBFDFtangent - \vw_{\RBFDFtangent}$ for $(\vf,g)\in \calB$.
\end{proof}

Now we are ready to prove one of our main results. 
\begin{theorem}\label{dftanthm}
Let $0\leq \mu\leq \tau$. Given $\vf\in \vH^{\tau}(\Omega)$ and admissible $g\in H^{\tau-1/2}(\Gamma)$, we denote the decomposition of $\vf$ from Proposition \ref{hodgeprop} as $\vf = \vw_\vf + \nabla p_\vf$. Then we have
\begin{eqnarray*}
\| P_{div}\RBFDFtangent - \vw_\vf\|_{\vH^\mu(\Omega)} \leq C \left(h_{\Omega}^{\tau-\mu} + h_\Gamma^{\tau-\mu}\right)\left(\| \vf\|_{\vH^{\tau}(\Omega)} +\|g\|_{H^{\tau - 1/2}(\Gamma)}\right).
\end{eqnarray*}
\end{theorem}

\begin{proof}
 We begin with a triangle inequality and an application of Lemma \ref{lemma:projapprox}:
\begin{eqnarray*}
\|P_{div}\RBFDFtangent - \vw_\vf\|_{\vH^\mu(\Omega)} & \leq &\|\vw_{\RBFDFtangent} - \vw_\vf\|_{\vH^\mu(\Omega)} +  C h_\Gamma^{\tau-\mu}\left(\| \vf\|_{\vH^{\tau}(\Omega)} + \|g\|_{\vH^{\tau - 1/2}(\Gamma)}\right).
\end{eqnarray*}
Next we bound $\|\vw_{\RBFDFtangent} - \vw_\vf\|_{\vH^\mu(\Omega)}$. Note that $\RBFDFtangent - \vf = (\vw_{\RBFDFtangent} - \vw_\vf) + \nabla (p_{\RBFDFtangent} - p_\vf)$ decomposes $\RBFDFtangent - \vf$ as in Proposition \ref{hodgeprop} with $g = 0$. Applying Proposition \ref{hodgeprop_potentials} to $\vf - \RBFDFtangent$, we get that $\vw_{\RBFDFtangent} - \vw_\vf = \vcurl(\vpsi)$ with $\vpsi$ satisfying \eqref{potentialregularity}, which yields:
\[\|\vw_{\RBFDFtangent} - \vw_\vf\|_{\vH^\mu(\Omega)} = \|\vcurl(\vpsi)\|_{\vH^{\mu+1}(\Omega)} \leq C \|\RBFDFtangent - \vf\|_{\vH^\mu(\Omega)}.\]
An application of Lemma \ref{lemma:fullapprox} finishes the proof.
\end{proof}

Since $P_{curl}\RBFDFtangent - \nabla p_\vf = \RBFDFtangent - \vf + \vw_\vf - P_{div}\RBFDFtangent$, similar estimates hold for the curl-free part.

\subsection{Convergence with Curl-free Boundary Conditions}\label{cferrorsection}

Now we focus on the decomposition in Proposition \ref{hodgeprop2}. Recall that there is a projector $\Pcfn$ that projects $\vf$ onto the curl-free term in this decomposition, and that $\RBFCFnormal$ denotes the kernel interpolant from Section \ref{cfdecomp} whose tangential components of $P_{curl}\RBFCFnormal$ are forced to vanish on the node set $Y\subset \Gamma$. Showing that $P_{curl}\RBFCFnormal$ approximates $\Pcfn\vf$ uses arguments similar to those in the preceeding section, thus we provide only the aspects of the proof that are significantly different

First, we have a lemma, whose proof we omit since the arguments are similar to those of Lemma \ref{lemma:fullapprox} - the most major difference here is that the proof requires an extension $E$ so that $\RBFCFnormal = \vs^\vn_{E\vf}$, and such an extension exists by Lemma 1 and the remark proceeding it. 
\begin{lemma}\label{lemma:cffullapprox}
Let $\mu$ satisfy $0\leq \mu\leq \tau$. Then for all $\vf\in \vH^\tau(\Omega)$ we have
\[ \|\vf - \RBFCFnormal\|_{\vH^\mu(\Omega)} \leq C h_{\Omega}^{\tau-\mu}\| \vf\|_{\vH^\tau(\Omega)}.\]
\end{lemma}

Next we have a lemma analogous to Lemma \ref{lemma:projapprox}. 

\begin{lemma}
Let $0\leq \mu\leq \tau$. Then for all $\vf\in \vH^\tau(\Omega)$ we have
\[\|\Pcfn \RBFCFnormal - P_{curl}\RBFCFnormal\|_{\vH^\mu(\Omega)} \leq C h_{\Gamma}^{\tau-\mu} \|\vf\|_{\vH^\tau(\Omega)}.\]
\end{lemma}

\begin{proof}
We will use the tangential trace operator $\gamma_\vt$, which is defined on smooth vector fields as $\gamma_\vt \vv:= \vv|_\Gamma \times \vn$. By \cite[Theorem 2.11, page 34]{GiraultRaviart1986}, this extends to a continuous map defined on $\vL_2(\Omega)$ vector fields with bounded curl (in $\vL_2$) to the space $\vH^{-1/2}(\Gamma)$, and the following Green's formula holds:
\begin{equation}\label{tangentialgreens}
(\vcurl(\vv),\vg) - (\vv,\vcurl(\vg)) = \langle \gamma_\vt \vv,\vg\rangle \quad \forall\, \vg\in \vH^1(\Omega).
\end{equation}

The first step is to transfer the problem to the boundary by showing that
\begin{equation}\label{bdrytransfer}
\|\Pcfn \RBFCFnormal - P_{curl}\RBFCFnormal \|_{\vH^\mu(\Omega)}  \leq C\|\gamma_\vt P_{curl}\RBFCFnormal\|_{\vH^{\mu-1/2}(\Gamma)}.
\end{equation}
For brevity, we let $\vv =\Pcfn \RBFCFnormal - P_{curl}\RBFCFnormal$. The identity 
\[
\Pcfn \RBFCFnormal - P_{curl}\RBFCFnormal = P_{div} \RBFCFnormal - \Pcfn^\perp \RBFCFnormal
\] 
implies that $\vv \in \vcurl(\vH^1(\Omega))$, so by Proposition \ref{boundedvelpot} $\vv$ has a potential $\vpsi$ satisfying 
\[\|\vpsi\|_{\vH^1(\Omega)} \leq C \|\vv\|_{\vL_2(\Omega)}.\]
With this, we can apply \eqref{tangentialgreens} with $\vg = \vpsi$ to get the inequality
\begin{eqnarray*}
\|\vv\|^2_{\vL_2(\Omega)}& =& |\langle \gamma_\vt \vv, \vpsi\rangle |\leq \|\gamma_\vt \vv\|_{\vH^{-1/2}(\Gamma)} \|\vpsi\|_{\vH^{1/2}(\Gamma)}\\
&\leq& \|\gamma_\vt \vv\|_{\vH^{-1/2}(\Gamma)} \|\vpsi\|_{\vH^{1}(\Omega)} \leq C\|\gamma_\vt \vv\|_{\vH^{-1/2}(\Gamma)} \|\vv\|_{\vL_2(\Gamma)}.
\end{eqnarray*}
Since $\gamma_\vt\Pcfn\RBFCFnormal = \mathbf{0}$, we obtain \eqref{bdrytransfer} when $\mu=0$:
\begin{equation}\label{vbound}
\|\vv \|_{\vL_2(\Omega)} \leq C\|\gamma_\vt P_{curl}\RBFCFnormal\|_{\vH^{-1/2}(\Gamma)}.
\end{equation}
For $\mu\geq 1$, we use \eqref{equivsobnorm_tan}. Using \eqref{vbound} and the fact that $\vv$ is both divergence-free and curl-free, we get 
\[
\|\vv\|_{\vH^\mu(\Omega)}^2 \leq C \vertiii{\vu}^2_\vt =C (\|\vv\|_{\vL_2(\Omega)}^2 + \|\gamma_\vt\vv\|_{\vH^{\mu-1/2}(\Gamma)}^2) \leq C\|\gamma_\vt \vv\|_{\vH^{\mu-1/2}(\Gamma)}^2 = C\|\gamma_\vt P_{curl}\RBFCFnormal\|_{\vH^{\mu-1/2}(\Gamma)}^2.
\]
This proves \eqref{bdrytransfer} for $\mu=0$ and $1\leq \mu \leq \tau$. By design $\gamma_\vt P_{curl}\RBFCFnormal$ has many zeros on $\Gamma$, which makes this situation very similar to that in Lemma \ref{lemma:bdryapprox}, whose arguments can be repeated to arrive at the bound
\[
\|\gamma_\vt P_{curl}\RBFCFnormal\|_{\vH^{\mu-1/2}(\Gamma)}\leq C h_\Gamma^{\tau-\mu}\|\vf\|_{\vH^\tau(\Omega)}.
\]
The case $0<\mu<1$ can now be handled by operator interpolation. This finishes the proof.
\end{proof}

With these results, one can now construct an argument very similar to the proof of Theorem \ref{dftanthm} to arrive at the result below, which we state without proof. 
\begin{theorem}\label{cfnthm}
Let $0\leq \mu\leq \tau$. Then for all $\vf \in \vH^\tau(\Omega)$ we have
\begin{eqnarray*}
\|\Pcfn\vf - P_{curl}\RBFCFnormal\|_{\vH^\mu(\Omega)} \leq C \left(h_{\Omega}^{\tau-\mu} + h_\Gamma^{\tau-\mu}\right)\| \vf\|_{\vH^{\tau}(\Omega)}.
\end{eqnarray*}
\end{theorem}

\begin{remark}
We heavily relied on the fact that given $\vf\in \vH^{\tau}(\Omega)$, we are guaranteed potential functions having the appropriate smoothness (assuming $\Gamma$ is smooth enough). We are not aware of such a result for functions in native spaces associated with $C^\infty$ kernels, even for very smooth domains. However, convergence results for the decompositions treated here can be derived for $C^\infty$ kernels, assuming that all potentials (or their components) reside within $\mathcal{N}_\phi$, where $\Phi = -\Delta\phi$. 
\end{remark}

\section{Numerical Examples}\label{numerics}

In this section we illustrate the methods described previously with numerical experiments. We start with the following target function:
\begin{equation}\label{targetfun}
\vf = \vcurl(\cos(2(x^2 + y^2))) + \nabla p,
\end{equation}
where $p$ is the \texttt{MATLAB} \emph{peaks} function, and consider $\vf$ on the annulus $\Omega$ centered at the origin with inner radius $.75$ and outer radius $2$ (see Figure \ref{Targetannulus}). This function on $\Omega$ has the property that the Leray projection, $P_L\vf$, is equal to $\vcurl(\cos(2(x^2 + y^2)))$, and in what follows we will compare $P_L\vf$ to $P_{div}\RBFDFtangent$.
\begin{figure}
\centering
\begin{tabular}{cc}
\subfigure[Target Field on $\Omega$]{\includegraphics[width=0.37\textwidth]{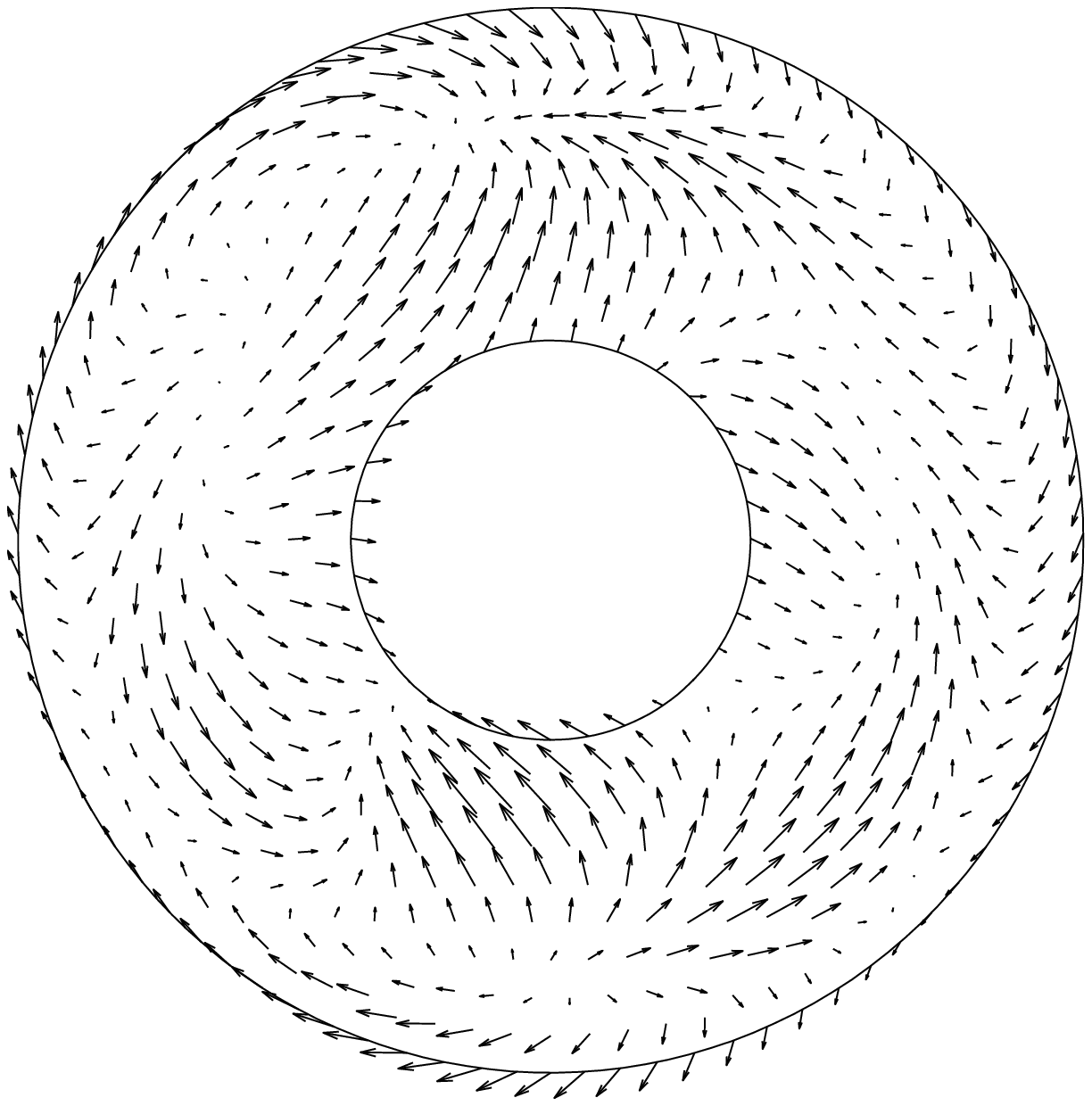}\label{Targetannulus}} &
\subfigure[Example Node Layout]{\includegraphics[width=0.37\textwidth]{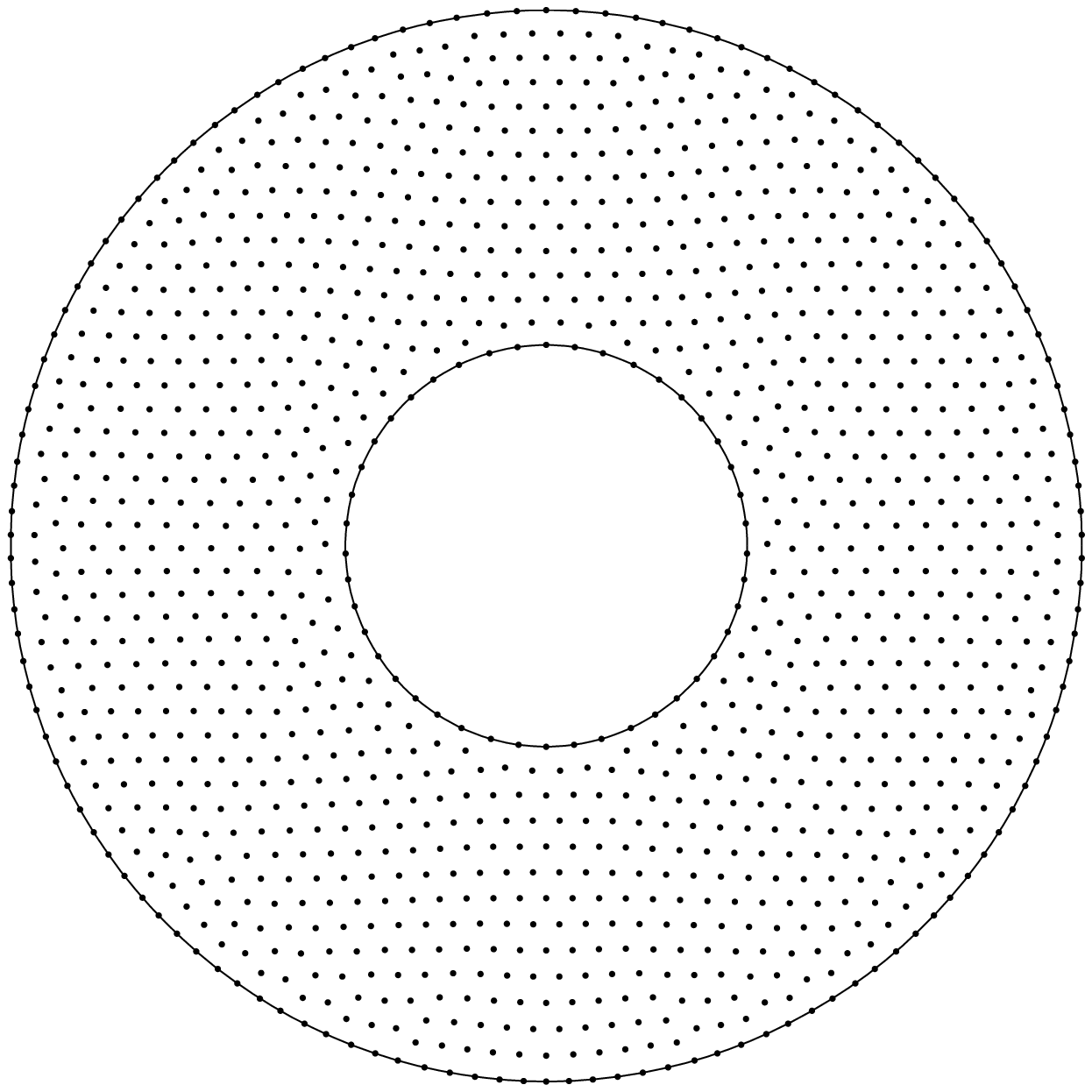}\label{annulusnodes}}  
\end{tabular}
\caption{The domain and target field $\vf$ used in the first experiment.}
\end{figure}
We used the freely available \texttt{distmesh} package to generate quasi-uniformly spaced nodes on $\Omega$ \cite{Persson_distmeshpaper} for the experiments. Eight nodes sets were generated with the number of full-interpolation centers ranging from $N=615$ to $N=11210$, and the number boundary centers ranging in cardinality from $M=115$ to $M=521$. An example node set with $N=1276$ is pictured in Figure \ref{annulusnodes}. In every experiment, we enforced full-interpolation at all centers, including the boundary sites. \texttt{MATLAB} files containing the nodes used and other useful files can be downloaded from \cite{FuselierHHDdata}. To generate our matrix-valued kernels, we used the scalar Mat\'ern kernel $\phi$ given by
\[\phi(r) = \frac{1}{945}e^{-r}(r^5 + 15r^4 + 105r^3 + 420 r^2 + 945r + 945),\]
where $r=r(x,y)=\epsilon \sqrt{x^2 + y^2}$. The free parameter $\epsilon$, known as the shape parameter, affects the stability and accuracy of the method. The shape parameter remained fixed at $\epsilon = 5$ throughout our experiments, which kept the computations relatively stable. The two dimensional version of this kernel, $\phi(\sqrt{x^2 + y^2})$, satisfies $\widehat{\phi}(\omega) = C (1 + |\omega|^2)^{-13/2}$, where $C$ is a constant, which means in particular that the matrix kernel $\Phi$ satisfies \eqref{algebraicdecay} with $\tau = 5.5$. 

We measured the relative error $\|P_{div}\RBFDFtangent - P_L\vf \|_{\ell_2(X)}/\|P_L\vf\|_{\ell_2(X)}$, where $X$ is the finest node set of those described above (i.e. with $\# X = 11210$) and the norm is given by
\[\|\vg\|_{\ell_2(X)} = \sqrt{\frac{1}{\# X}\sum_{x_j\in X} |\vg(x_j)|^2}.\]
The error between the generalized interpolant $\RBFDFtangent$ and $\vf$ was recorded similarly.
Lemma \ref{lemma:fullapprox} and Theorem \ref{dftanthm} dictate that the $L_2(\Omega)$ errors should all decay like $\calO(h^{5.5})$. Since our nodes are very uniform, $\|\cdot\|_{\ell_2(X)}\sim \|\cdot\|_{L_2(\Omega)}$, so observing $\calO(h^{5.5})$ would confirm these results. Due to the quasi-uniformity of the nodes, the mesh norm $h$ behaves asymptotically like $1/\sqrt{N}$, where $N$ is the number of nodes in a given node set. A loglog plot of error versus $1/\sqrt{N}$ is given in Figure \ref{convergence}, where it can be seen that the error for the Leray projection appears to converge slightly faster than $\calO(h^{5.5})$.
\begin{figure}
\centering
\begin{tabular}{ccc}
\subfigure[$P_{div}\RBFDFtangent = \vcurl(\RBFstreamtangent)$ (Leray Projection)]{\includegraphics[width=0.37\textwidth]{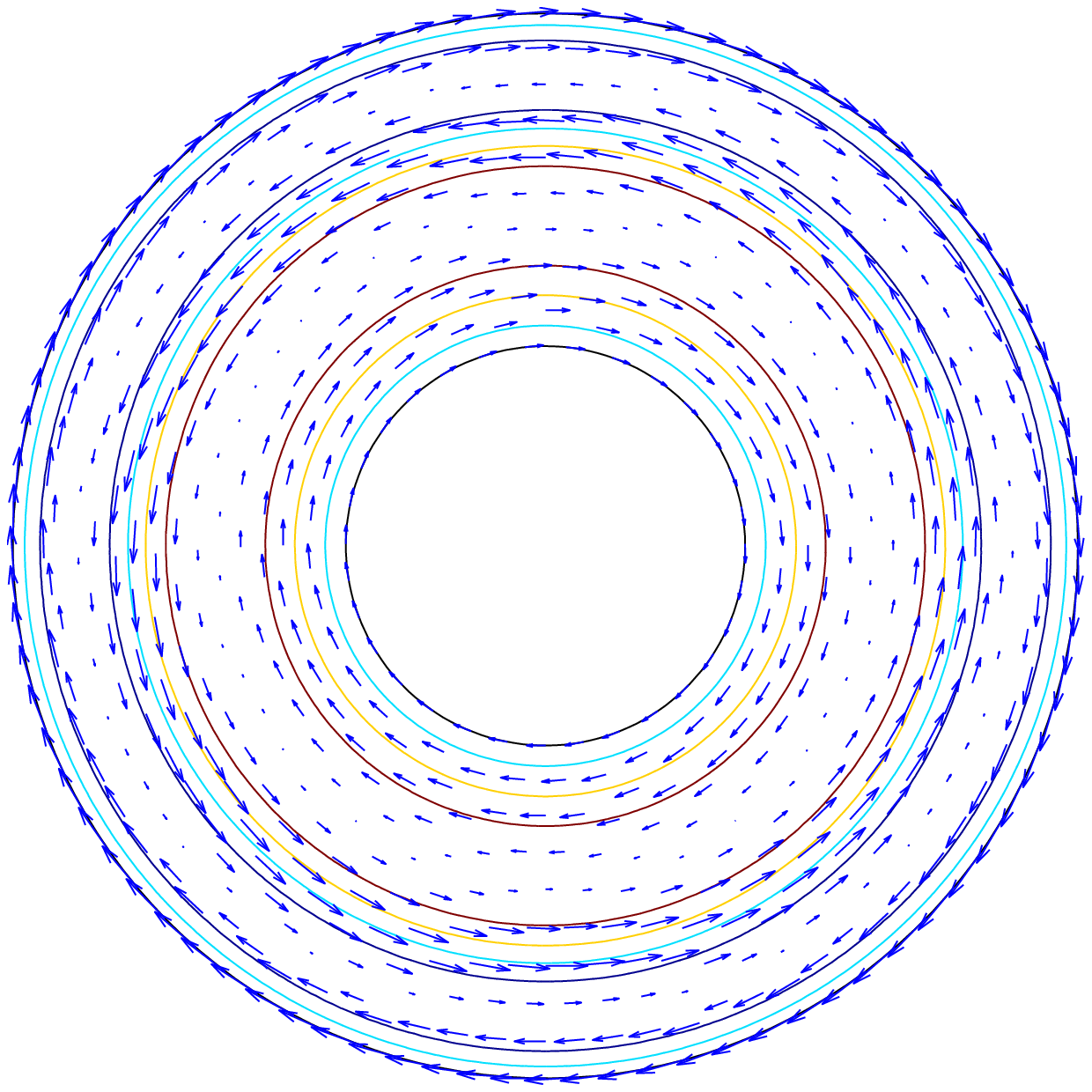}\label{RBFleray_annulus}} &
\subfigure[$P_{curl}\RBFDFtangent=\nabla\RBFpresstangent$]{\includegraphics[width=0.37\textwidth]{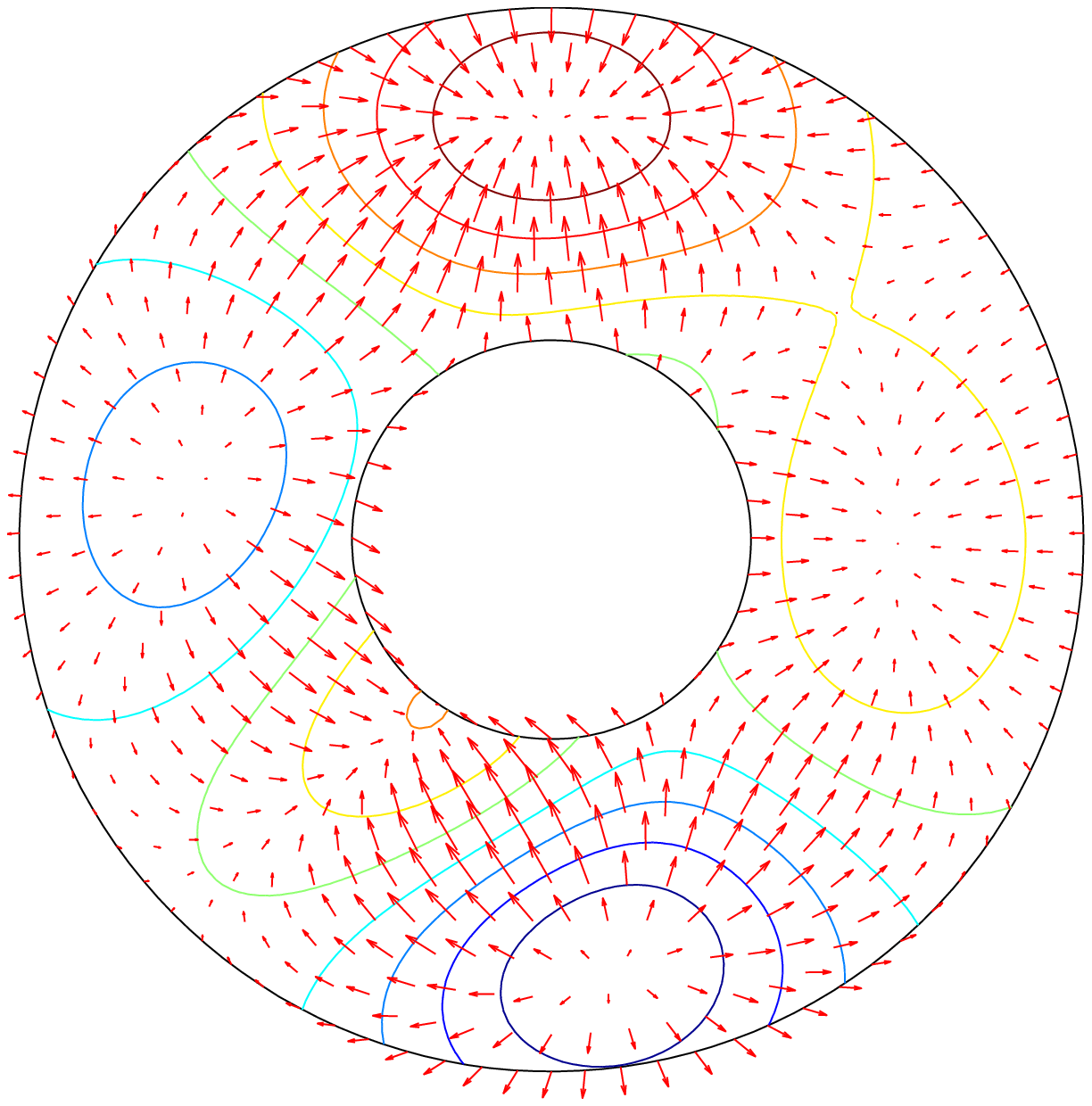}\label{RBFcf_annulus}}  
\end{tabular}
\caption{The kernel decomposition of $\vf$ using $\RBFDFtangent = P_{div}\RBFDFtangent + P_{curl}\RBFDFtangent$. The contours represent the potentials $\RBFstreamtangent$ and $\RBFpresstangent$.}
\end{figure}

%
In the next experiment, we computed the full Helmholtz-Hodge decomposition (HHD) of $\vf$ on a slightly more complicated domain, and in the process obtained evidence for the bound in Theorem \ref{cfnthm}. Recall that the full HHD is given by
\begin{equation}\label{fulldecomp}
\vf = \Pcfn\vf + P_L\vf + \nabla h,
\end{equation}
where $\Pcfn\vf$ is the curl-free normal component of $\vf$ from Proposition \ref{hodgeprop2}, $P_L\vf$ is the Leray projection, and $h$ is a harmonic function. We used the same target function \eqref{targetfun}, but on the domain pictured in Figure \ref{cfnormfieldplot}. As in the previous test, several quasi-uniform node sets were generated using the \verb1distmesh1 package with sizes ranging from $N=486$ to $N=16882$ (see \cite{FuselierHHDdata}). Samples of $\vf$ at these sites were used to obtain approximations to each term in \eqref{fulldecomp} using the method described below.
\begin{figure}
\centering
\begin{tabular}{cc}
\subfigure[Convergence for the Annulus Experiment]{\includegraphics[width=0.4\textwidth]{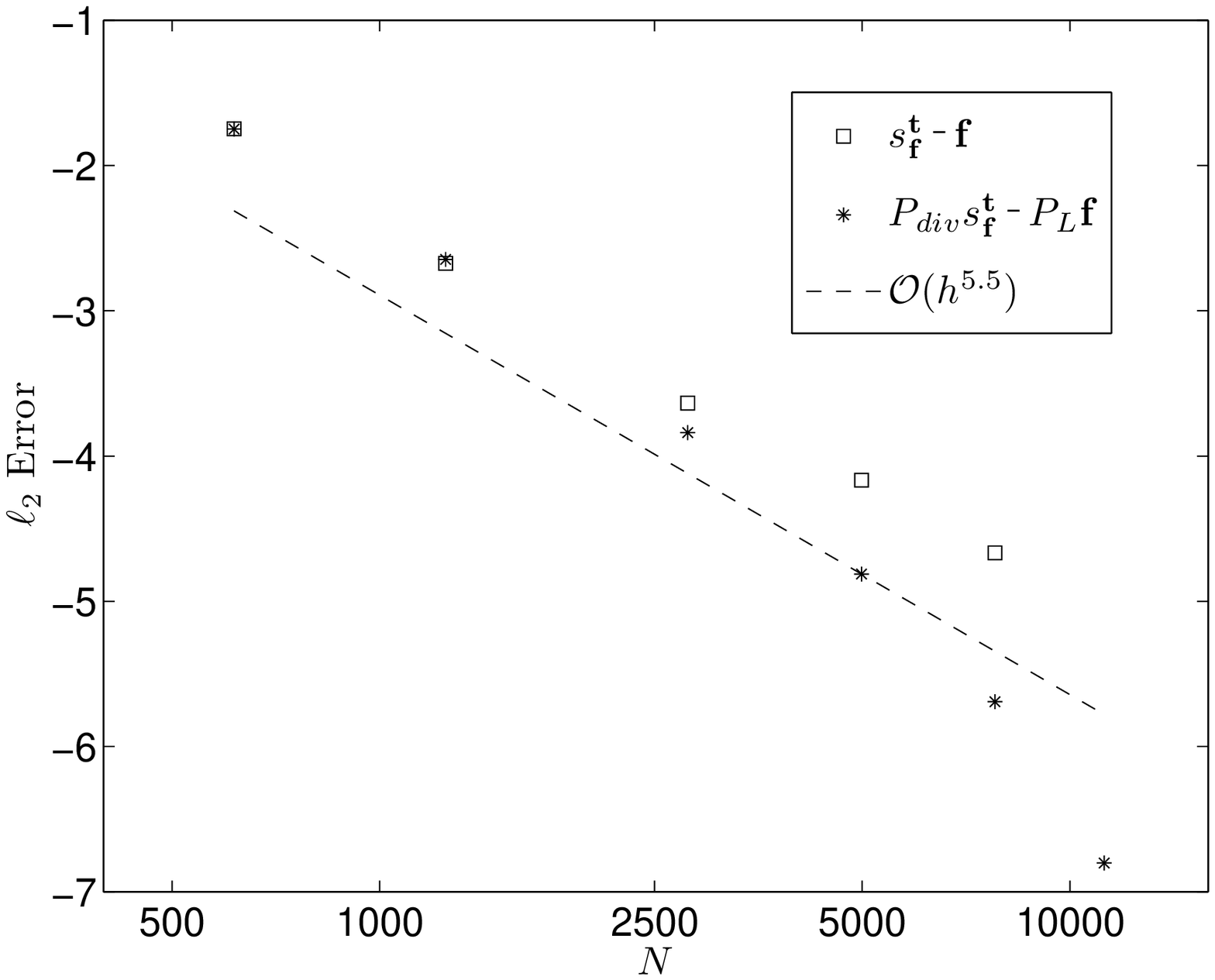}\label{convergence}} &
\subfigure[Convergence for the Full HHD Experiment]{\includegraphics[width=0.4\textwidth]{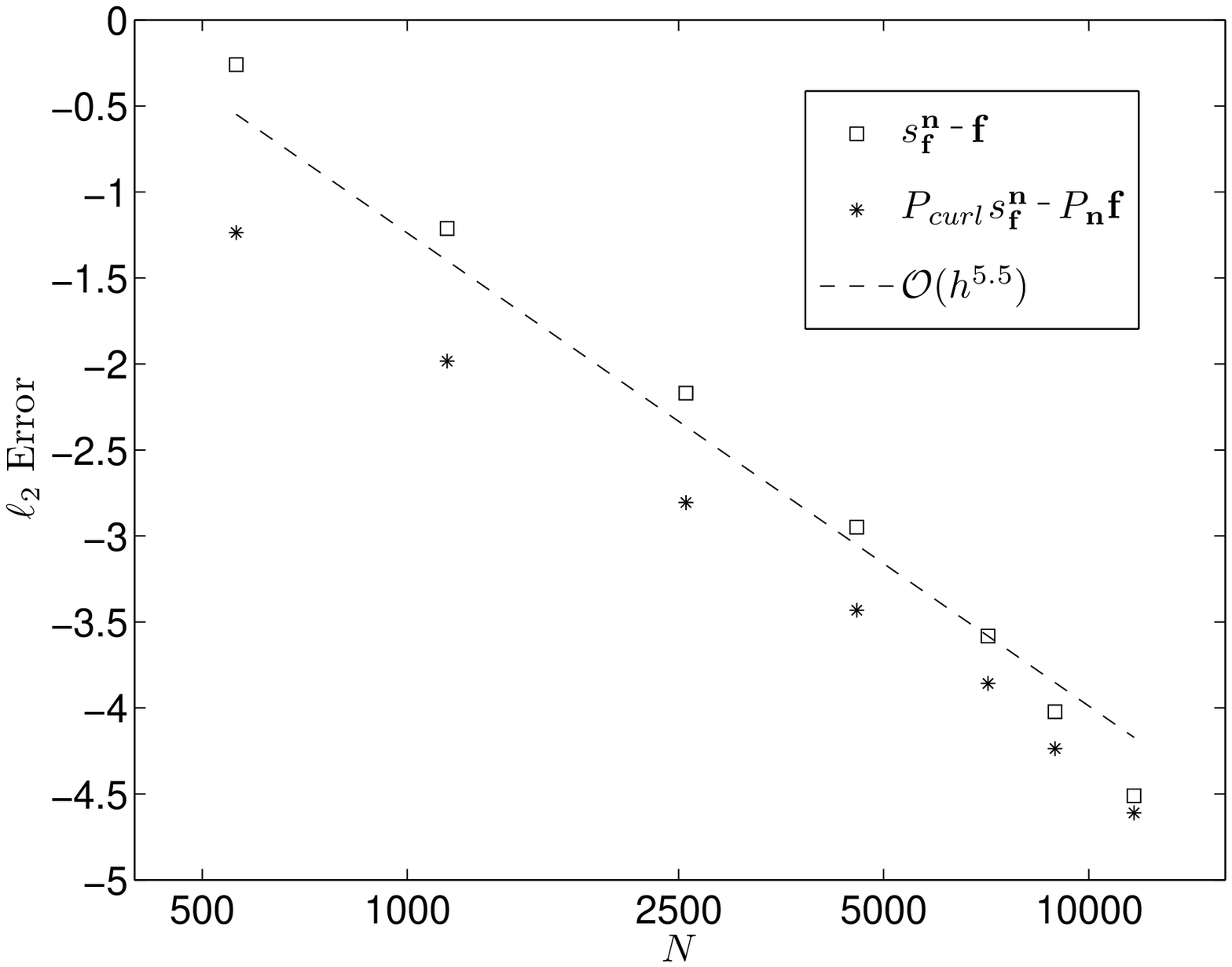}\label{convergence_tplannulus}}
\end{tabular}
\caption{Convergence results for each numerical experiment. The vertical axis gives the logarithm of the relative $\ell_2(X)$ error (base 10), and the horizontal axis gives $N$ on a $\log_{10}$ scale.} 
\end{figure}

The first step of the two-step process is to construct an interpolant of $\vf$ with curl-free boundary conditions of form \eqref{eq:curlfreenormalform} that solves the system \eqref{cfsystem}. Let $\RBFCFnormal$ denote this interpolant and note that $P_{curl}\RBFCFnormal$ approximates $\Pcfn\vf$. Second, decompose $P_{div}\RBFCFnormal$ to approximate $P_L\vf$ and $\nabla h$ by using an interpolant with divergence-free boundary conditions of the form \eqref{eq:divfreeform} that solves \eqref{system} (with $g=0$ and $\vf$ replaced by $P_{div}\RBFCFnormal$). Denote this interpolant by $\RBFDFtangent$, and note that $P_{div}\RBFDFtangent\sim P_L\vf$ and $P_{curl}\RBFDFtangent\sim \nabla h$.  These steps give approximations to the three components of the decomposition of $\vf$, which are plotted in Figure \ref{fullhodge}, together with contour plots of the corresponding potential functions. 
\begin{figure}%
\centering
\begin{tabular}{lll}
\subfigure[Curl-Free Normal Portion]{\includegraphics[width=0.27\textwidth]{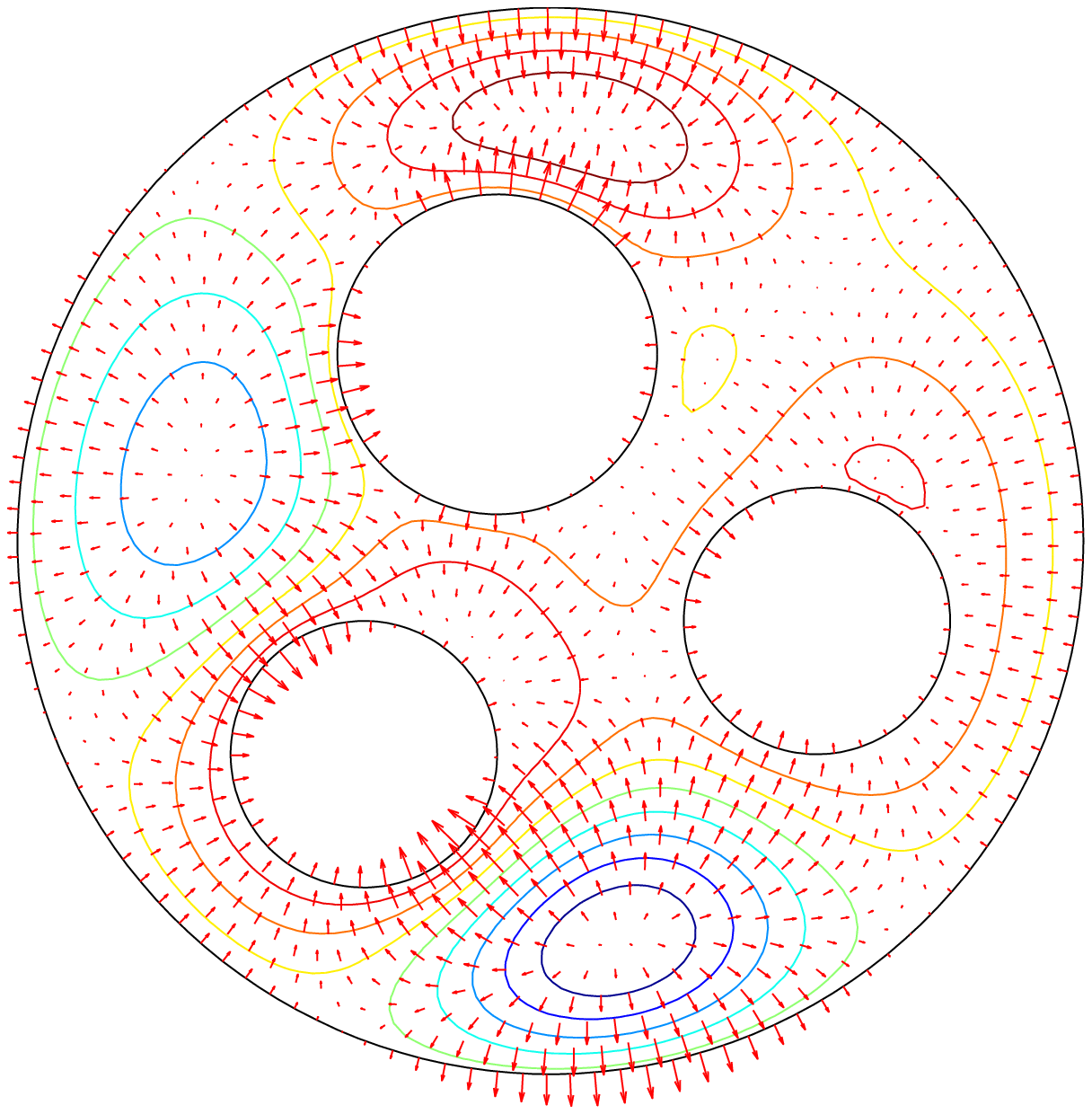}\label{cfnormfieldplot}} &
\subfigure[Leray Projection]{\includegraphics[width=0.27\textwidth]{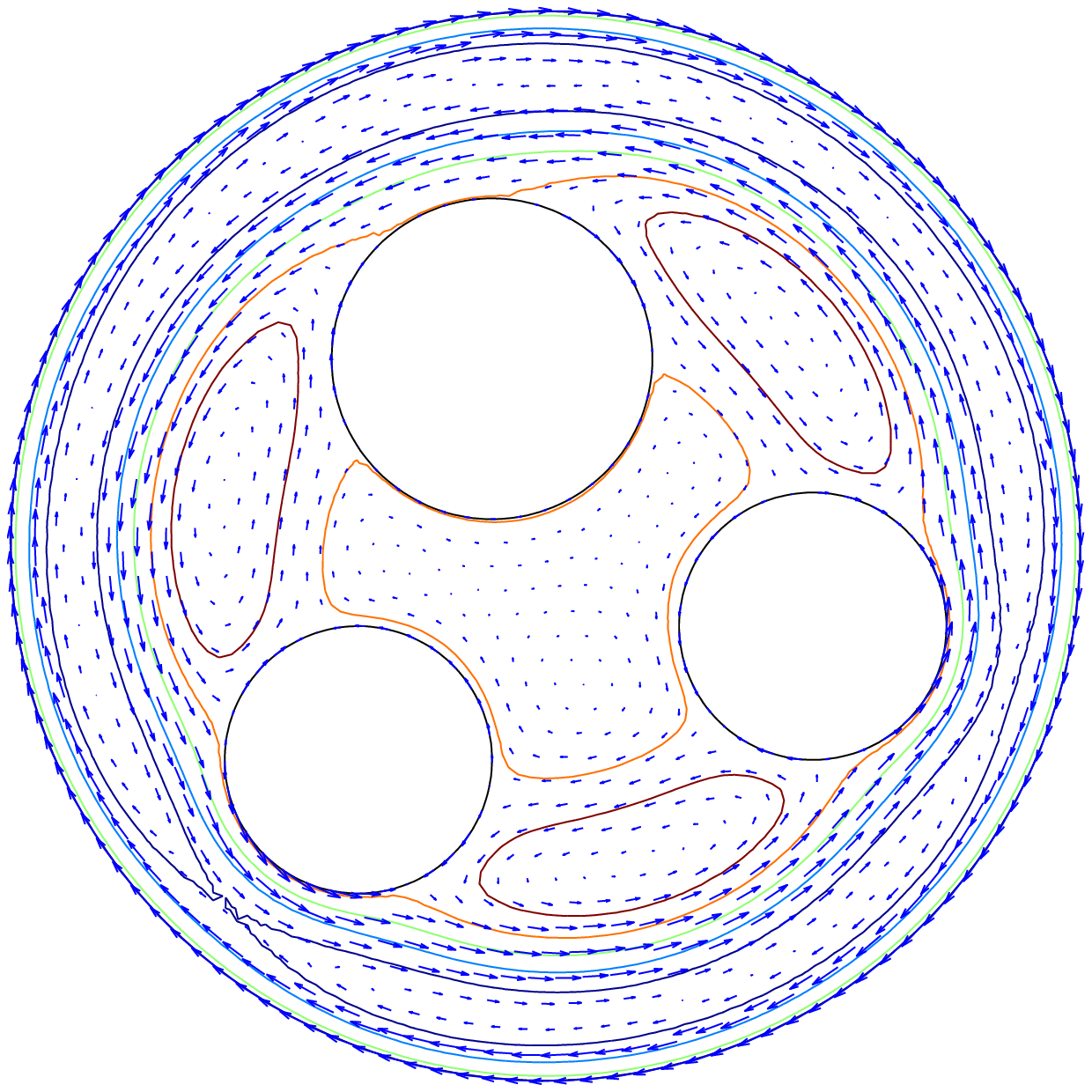}\label{lerayfieldplot}} &
\subfigure[Harmonic Portion]{\includegraphics[width=0.27\textwidth]{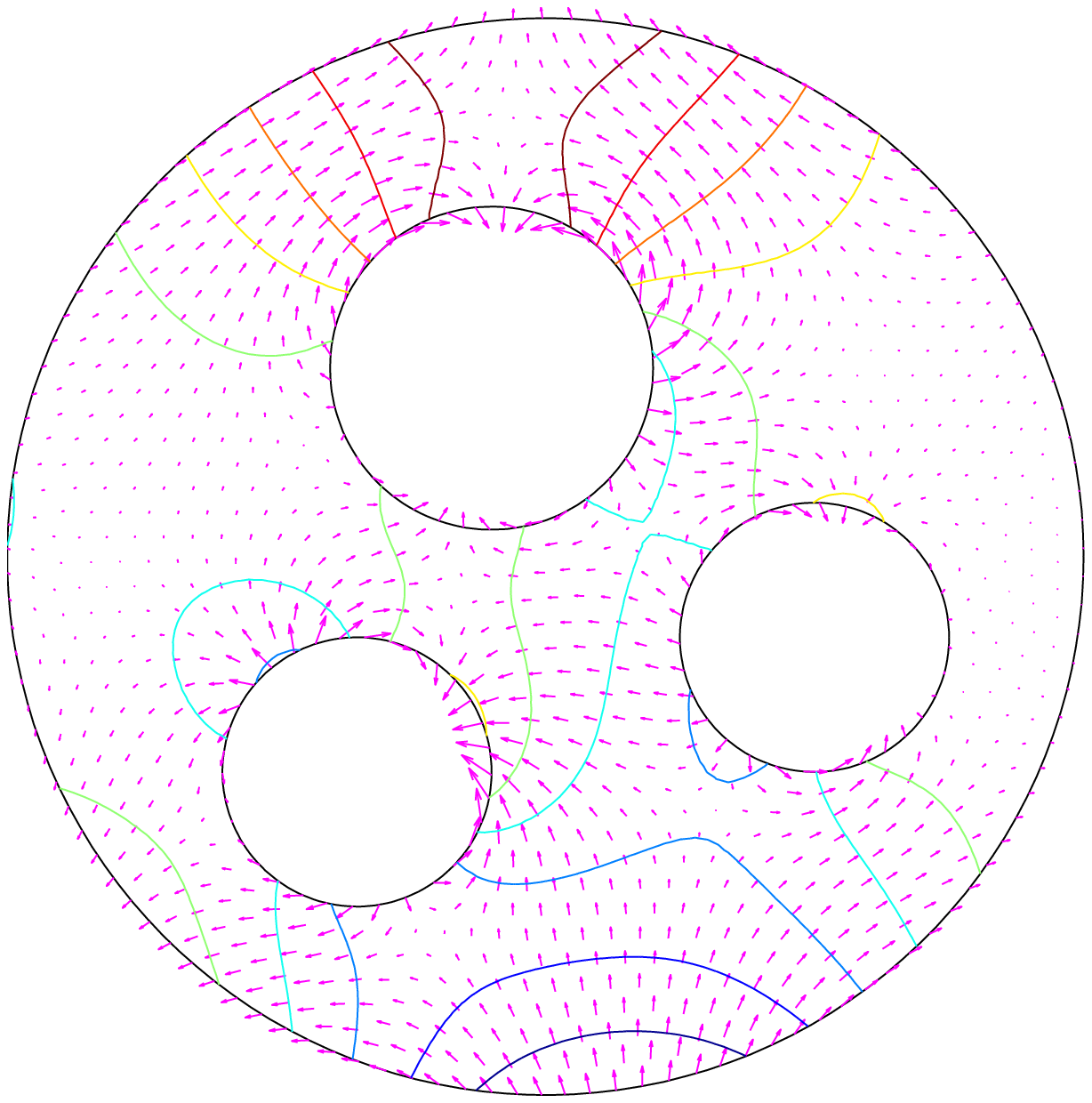}\label{harmonicfieldplot}} \\
\end{tabular}
\caption{The kernel approximation of the full HHD for the target field $\vf$ \eqref{targetfun}, with contours of each term's scalar potential.}\label{fullhodge}
\end{figure}

With regard to convergence, we did not measure the error directly because the exact decomposition for $\vf$ on this domain is unknown to us. Nevertheless, we estimated the rate of convergence by using each approximation on the finest node set as proxies for the true solution. To measure the error corresponding to $\Pcfn\vf$, for example, we used $\|P_{curl}\RBFCFnormal - \nabla p \|_{\ell_2(X)}$ where $\nabla p$ is the kernel approximation to $\Pcfn\vf$ on the finest node set $X$ (with $\# X = 16882$). We also tested the error between the generalized interpolant $\RBFCFnormal$ and $\vf$. Lemma \ref{lemma:fullapprox} and Theorem \ref{cfnthm} dictate that the $L_2(\Omega)$ errors should all decay like $\calO(h^{5.5})$. A loglog plot of error versus $1/\sqrt{N} \sim h$ is given in Figure \ref{convergence_tplannulus}, where the errors seem to be converging like $\calO(h^{5.5})$.
%
%
\section{Concluding Remarks}

Decompositions with other boundary conditions are certainly also possible. If no boundary conditions are specified, one can find an interpolant $\vs_\vf$ using only shifts of positive definite kernel $\Phi= -\Delta\phi I$. Enforcing $\vs_\vf|_X = \vf|_X$ leads to a positive definite system, and since $\Phi=\Phi_{div} + \Phi_{curl}$, $\vs_\vf$ decomposes trivially. This idea was used in a decomposition technique using thin plate splines introduced in earlier work \cite{AmodeiBenbourhim_1991}.
For other boundary conditions, if the functionals associated with the interpolation and boundary conditions are linearly independent and the Reisz representers are chosen as basis functions, then the kernel decomposition can be constructed. In this way, one could impose a whole host of boundary conditions in vector decomposition problems, and do so in a natural way.

\bibliographystyle{abbrv}


\bibliography{RBFvectordecompositions_arxiv_03052015}

\end{document}